\documentclass[11pt,a4paper]{article}
\usepackage[utf8]{inputenc}
\usepackage[T1]{fontenc}
\usepackage{lmodern}
\usepackage{microtype}
\usepackage[margin=2.6cm]{geometry}
\usepackage{amsmath,amssymb,amsthm,mathtools,mathrsfs}
\usepackage{xcolor}
\usepackage{enumitem}
\usepackage[colorlinks=true,linkcolor=blue!50!black,citecolor=green!40!black,urlcolor=blue!60!black]{hyperref}
\usepackage{comment}
\usepackage[backend=biber,sorting=nyt,style=ieee]{biblatex}
\definecolor{LightGray}{rgb}{0.83,0.83,0.83}
\RequirePackage[obeyFinal,
                color       = LightGray,
                bordercolor = LightGray,
                textsize    = footnotesize,
                figwidth    = 0.99\linewidth]{todonotes}

\theoremstyle{plain}
\newtheorem{theorem}{Theorem}[section]
\newtheorem{proposition}[theorem]{Proposition}
\newtheorem{lemma}[theorem]{Lemma}
\newtheorem{corollary}[theorem]{Corollary}

\theoremstyle{definition}
\newtheorem{definition}[theorem]{Definition}
\newtheorem{example}[theorem]{Example}

\theoremstyle{remark}
\newtheorem{remark}[theorem]{Remark}
\newtheorem{convention}[theorem]{Convention}

\newcommand{\R}{\mathbb{R}}
\newcommand{\Z}{\mathbb{Z}}
\newcommand{\T}{\mathbb{T}}
\newcommand{\N}{\mathbb{N}}
\newcommand{\Cx}{\mathbb{C}}
\newcommand{\hull}{\Omega}
\newcommand{\trans}{\Xi}
\newcommand{\GL}{\mathcal{G}_{\!L}}
\newcommand{\Fell}{\mathscr{C}}
\newcommand{\Cst}{C^*}
\newcommand{\Bd}{\mathbb{B}}
\newcommand{\co}{\mathrm{c}}
\newcommand{\dd}{\,\mathrm{d}}

\DeclareMathOperator{\dom}{dom}

\author{Diego Jaur\'e}
\date{}

\begin{document}

\title{\bfseries Wiener's lemma for quasicrystals: 
hull-uniform decay of covariant kernels on Delone sets}
\maketitle

\begin{abstract}
This paper studies continuous covariant kernel families over hulls of
Delone sets of finite local complexity using graded Wiener-algebra methods.
The norm given by the least common weighted summable off-diagonal
profile is identified isometrically with a graded $\ell^1$-norm.
The resulting Banach $*$-algebras are symmetric, spectrally invariant
in the reduced groupoid C*-algebra, and stable under holomorphic
functional calculus. Exact physical weights give preservation of
polynomial and stretched-exponential classes, including a fixed
stretched-exponential rate.

A configuration with dense orbit has a faithful regular representation.
Invertibility on that configuration therefore gives an inverse family
with one profile valid throughout the hull; for repetitive sets every
configuration has this property.

For polynomial weights of order $s>0$, return-vector counting yields
a differential estimate with parameter $\kappa/(\kappa+s)$, where
$\kappa$ is a growth exponent of the return vectors and may be taken
equal to $d$ for Meyer sets. Combined with established
norm-controlled inversion criteria, this gives explicit bounds in
the common-profile norm. Explicit bounds are also obtained for
stretched-exponential physical weights and for gapped Fermi projections.

\end{abstract}


\tableofcontents

\section{Introduction}\label{sec:intro}

\subsection*{The problem}

In localized-basis descriptions of electronic structure, Hamiltonians,
overlap matrices and density matrices are indexed by atomic positions.
Under suitable locality assumptions, their off-diagonal decay is relevant
to both analysis and approximation. Nearsightedness provides a physical
motivation for studying localization \cite{Kohn96,PK05}; spectral
invariance provides an algebraic mechanism for preserving specified
decay classes under inversion and holomorphic functional calculus
\cite{Groechenig}. These are related perspectives, not equivalent
statements without additional hypotheses. Inverse-closed matrix and
convolution-dominated algebras, including weighted and operator-valued
versions, have a substantial literature
\cite{GL,FGL08,BK,Baskakov97,KB25}.

For aperiodic point sets, a natural object is a continuous covariant
family $(K_\omega)_\omega$ over a compact hull
\cite{BHZ,Kellendonk,LenzStollmann03}.
Contour integration, approximation theory and Combes--Thomas methods
already provide localization estimates, including uniform estimates
under appropriate geometric and spectral assumptions
\cite{BBR,CO,OTC,MLO}. The family norm of interest requires one
weighted summable profile to dominate every kernel throughout the
hull. A uniform bound on the row and column sums is not, by itself,
a bound in this common-profile norm.

This paper identifies this norm with the graded Wiener norm of the
Delone groupoid algebra and makes the resulting inversion estimates
explicit in the geometry of the return vectors. For Delone sets of
finite local complexity, the construction treats polynomial weights
and stretched-exponential physical weights without changing the
prescribed stretched-exponential rate.

\subsection*{Main results}

Let $L\subset\R^d$ be an $(r,R)$-Delone set of finite local complexity
(FLC), $\hull$ its hull, $\trans\subset\hull$ the transversal of
configurations containing the origin. A \emph{hull-covariant kernel} is
a family $K=(K_\omega)_{\omega\in\trans}$ of functions
$K_\omega:\omega\times\omega\to M_n(\Cx)$ ($n$ orbitals per site)
which is invariant under simultaneous translation of configuration and
positions and depends continuously on the configuration in the local
matching topology (Definition~\ref{def:cd}). Tight-binding Hamiltonians
and overlap matrices whose entries are computed from the local atomic
environment are of this kind by construction. For a weight $w$ on
$[0,\infty)$, $\mathcal E^w_{\mathrm{cov}}(L)$ denotes the space of covariant
kernels with
\[
\|K_\omega(x,y)\|\le h(x-y),\qquad \sum_{z\in L-L}h(z)\,w(|z|)<\infty,
\]
the dominating profile $h$ being independent of $\omega$. The weights
of interest are the polynomial ones, $w(t)=(1+t)^s$, and the
stretched-exponential ones, $w(t)=e^{\lambda t^\alpha}$ with
$0<\alpha<1$; exponential weights are excluded, and must be
(Remark~\ref{rem:weights}). Every $K\in\mathcal E^w_{\mathrm{cov}}(L)$
acts as a bounded operator $\pi_\omega(K)$ on
$\ell^2(\omega)\otimes\Cx^n$ for each $\omega$.

\medskip\noindent\textbf{Theorem A} (Wiener's lemma for quasicrystals;
Corollary~\ref{cor:wiener}, Proposition~\ref{prop:dictionary},
Corollary~\ref{cor:weightclasses}, Proposition~\ref{prop:physicalweights}).
\emph{Let $L$ be repetitive and FLC, and let $w$ be polynomial or stretched-exponential.
Then $\mathcal E^w_{\mathrm{cov}}(L)$ is a unital Banach $*$-algebra
under kernel composition, inverse-closed in $\Bd(\ell^2(\omega)\otimes\Cx^n)$
for every $\omega$, and stable under holomorphic functional calculus.
Explicitly: if $K\in\mathcal E^w_{\mathrm{cov}}(L)$ and $\pi_\omega(K)$
is invertible for \emph{one} $\omega\in\trans$, then $\pi_{\omega'}(K)$
is invertible for \emph{every} $\omega'$, $K^{-1}$ belongs to
$\mathcal E^w_{\mathrm{cov}}(L)$, and a single profile dominates
$(K^{-1})_{\omega'}$ for all $\omega'$ simultaneously.}

\medskip
Theorem A delivers the statements that the electronic-structure
literature uses (\S\ref{sec:consequences}). For an overlap matrix
$S=\mathbf 1+A$ with $A\in\mathcal E^w_{\mathrm{cov}}(L)$, bounded below
by $\delta>0$ on one configuration, $S$ is bounded below on every
configuration and $S^{-1},S^{1/2},S^{-1/2}\in\mathcal E^w_{\mathrm{cov}}(L)$
with hull-uniform profiles (Corollary~\ref{cor:overlap}): this is the
decay of the L\"owdin orthogonalization on a quasicrystalline point set.
For a self-adjoint Hamiltonian $H\in\mathcal E^w_{\mathrm{cov}}(L)$ and $T>0$, the
Fermi--Dirac density matrix $f_T(H)$ has hull-uniform decay in the same
class with no gap assumption. Both statements are stable under
sufficiently small self-adjoint perturbations in the weighted norm,
with positivity retained for $S$. They extend to a uniform magnetic
field in the magnetically covariant kernel class
(Remark~\ref{rem:magnetic}), and the framework covers
decorated Delone sets (Remark~\ref{rem:decorated}), hence
$\Z^d$-subshifts and random binary alloys (Example~\ref{ex:random}).

Theorem A asserts membership; it gives no bound on the profile of
$K^{-1}$. The following theorem makes the polynomial regime explicit;
Theorem~\ref{thm:stretchednci} gives norm control at a fixed
stretched-exponential rate as well.

\medskip\noindent\textbf{Theorem B} (Norm-controlled inversion and
Fermi-projection bounds; Theorem~\ref{thm:nci},
Corollary~\ref{cor:ferminci}, Proposition~\ref{prop:nonci}). \emph{Let
$L$ be FLC, $w(t)=(1+t)^s$ with $s>0$, and let $\kappa$ be a growth
exponent of the return vectors,
$\#\big((L-L)\cap\bar B_\rho\big)\le C_\kappa(1+\rho)^\kappa$; one may
always take $\kappa\le N$, the rank of the group generated by $L-L$,
and $\kappa=d$ when $L$ is a Meyer set. Then $\mathcal E^w_{\mathrm{cov}}(L)$ admits norm-controlled
inversion in the associated groupoid C*-algebra:
whenever $K$ is invertible there,
\[
\|K^{-1}\|_{\mathcal E^w}
\le
h_{L,s}\big(
\|K\|_{\mathcal E^w},
\|K^{-1}\|_{\mathfrak C}
\big),
\]
where the differential estimate has parameter
$\theta=\kappa/(\kappa+s)$.
If $L$ is repetitive, or if the chosen configuration has
dense orbit, the C*-norm of the inverse equals
$\|\pi_\omega(K)^{-1}\|$. In particular, if
$H=H^*\in\mathcal E^w_{\mathrm{cov}}(L)$ has a spectral gap of
half-width $\Delta$ around the Fermi level in
$\mathrm{sp}_{\mathfrak C}(H)$, its Fermi projection satisfies}
\[
\|P_\omega(x,y)\|\le B\,(1+|x-y|)^{-s}
\qquad(\omega\in\trans,\ x,y\in\omega),
\]
\emph{with $B$ explicit in
$(\|H\|_{\mathcal E^w},\Delta,s,\kappa,C_\kappa)$ and the
geometric constants of $L$. The weight cannot be removed: for
$L=\Z^d$, norm control fails at $s=0$ by Nikolski's theorem
\cite{Nik}.}

\medskip
Theorem B makes the return-vector geometry visible quantitatively.
For Meyer sets, the return-vector count allows $\kappa=d$,
so the parameter $\theta=d/(d+s)$ in the differential estimate here
has the same form as that obtained from the
corresponding lattice counting argument.
This comparison concerns the present estimate; it does not
assert optimality of the control function or equality with
the best bounds available in periodic settings. For general FLC sets the guaranteed exponent
degrades to the rank $N$ (Remark~\ref{rem:ncidichotomy}).

\subsection*{The method}

Both theorems are proved by exhibiting
$\mathcal E^w_{\mathrm{cov}}(L)$ as a subalgebra of the reduced
$C^*$-algebra $\mathfrak C=C^*_r(\GL)$ of the Delone groupoid
$\GL=\{(\omega,x):x\in\omega\in\trans\}$, and by proving symmetry and
spectral invariance there. Three ingredients enter.

\emph{The address grading.} By Lagarias \cite{Lagarias}, the return
vectors $L-L$ generate a free abelian group of finite rank $N\ge d$.
Choose an address isomorphism $a$ onto $\Z^N$. The cocycle
$c(\omega,x):=a(x)$ has compact open bisection fibres
(Proposition~\ref{prop:cocycle}), and gives a topological $\Z^N$-grading
of $\mathfrak C$, with homogeneous norms identified with supremum
norms (Theorem~\ref{thm:grading}). The unweighted spectral-invariance
statement is an application of \cite[Thms.~2.4 and~2.8]{JM}.
For the weighted statement, a short spectral-radius transfer
argument from scalar Beurling-algebra theory is used
\cite{FGL}; this avoids imposing the shell-comparability condition
of \cite[Thm.~2.20]{JM}. On $\Z^N$, ordinary GRS already implies
the uniform GRS condition for finite generating sets, so no independent
relaxation of that condition is claimed. Proposition~\ref{prop:weightedSI}
records the resulting form convenient for exact physical weights.

\emph{The kernel dictionary.} The bisection property is exactly what
makes the map $\Phi\mapsto K^\Phi$, $K^\Phi_\omega(x,y)=\Phi(\omega-x,y-x)$,
an \emph{isometric} $*$-isomorphism of $\ell^{1,\nu}(\mathfrak C)$ onto
the algebra of covariant kernels dominated by $k\circ a$ with
$k\in\ell^{1,\nu}(\Z^N)$, the norm being the best dominating profile
(Proposition~\ref{prop:dictionary}). A configuration with dense orbit
has a faithful regular representation (Lemma~\ref{lem:densefaithful});
for repetitive $L$ this holds for every configuration, without requiring
aperiodicity or simplicity. Inverses therefore transport from a single
such configuration to the whole hull.

\emph{The metric comparison.} Lagarias's address characterization
of finite type \cite{Lagarias} compares address and Euclidean lengths
on return vectors. Explicit constants and their uniformity
over the hull are recorded in Lemma~\ref{lem:addcomp}.
Corollary~\ref{cor:weightclasses} identifies polynomial classes and
compares radial stretched-exponential scales, in general changing
the radial address rate. Instead, the exact pullback weight
$\widetilde\nu_w(g)=w(|\Psi g|)$ gives an isometric identification at
the prescribed physical rate
(Proposition~\ref{prop:physicalweights}). For quantitative estimates,
the supported degrees $a(L-L)$ are counted rather than all degrees of
$\Z^N$ whenever this gives a better exponent. The resulting differential
estimates permit application of the norm-controlled inversion criteria
of Samei--Shepelska \cite{SS}.

\subsection*{Relation to other work}

The algebraic starting point is the spectral-invariance theory for
topologically graded C*-algebras and partial actions developed in
\cite{JM}, together with the related Fell-bundle framework of
\cite{FJM,LP}. The kernel dictionary is a Delone realization of this
mechanism, with the bisection norms identifying the least
hull-uniform profile exactly. No new general unweighted Wiener theorem
for graded C*-algebras is claimed.

For infinite matrices, Baskakov \cite{Baskakov97},
Gr\"ochenig--Leinert \cite{GL}, and Fendler--Gr\"ochenig--Leinert
\cite{FGL08} provide relevant weighted and convolution-dominated
antecedents. K\"ohldorfer--Balazs \cite{KB25} treat operator-valued
matrix algebras on relatively separated sets; their
convolution-dominated construction in Section~5 is indexed by
$\Z^d$. The formulation here instead concerns continuous covariant
families with varying Delone index sets and a single summable
profile over all return vectors in the hull.

Norm-controlled inversion is also an established theory
\cite{GK2,GKII,SS,ShSu}. In particular, Flores
\cite[Theorem~1.4 and Section~5.4]{Flores25} proves it for polynomially
weighted cross-section algebras of Fell bundles over discrete groups.
Applied to the address Fell bundle, that result already supplies
existence of polynomial norm control. Here the additional information
is the explicit return-vector-dependent differential estimate and
its interpretation in the physical common-profile norm, rather than
existence of norm control for the first time.

For \'etale groupoids, Austad--Ortega \cite{AO} study hermitian
Banach $*$-algebras, while Austad--Ortega--Palmstr\o m
\cite{AOP26} develop weighted spectral invariance and $K$-theory
results under strong subexponential growth. Their
Lemmas~4.5--4.6 are a close methodological antecedent to the
stretched-exponential argument below. Their weighted fibre norms
place the supremum over units outside the fibre sum. The norm here
instead uses the sum of the degreewise suprema,
\[
  \|\Phi\|_{1,\nu}=\sum_g\nu(g)\|\widetilde P_g(\Phi)\|.
\]
At $p=1$ this dominates the corresponding weighted $I$-norm. The distinction is the order of
summation and supremum, not whether weights can be used.

Accordingly, the contribution is a quantitative Delone specialization:
an isometric common-profile dictionary, exact physical weights, and
inversion estimates governed by the growth of the supported return
vectors. No optimality claim is made for the control functions. The Fermi and
overlap conclusions are applications of this algebraic control,
complementary to localization estimates in \cite{BBR,CO,OTC}; they
are not claims of first localization results for aperiodic systems.
Likewise, the $K$-theory isomorphism follows from spectral invariance
and density. It concerns representatives in weighted algebras, not
a computation of new $K$-groups or gap-label sets
\cite{Bost,BHZ,Kellendonk,PSB}.

The paper does not treat general Fell-bundle coefficients over the
Delone groupoid, fixed exponential-weight spectral invariance,
infinite-local-complexity tilings, or optimal control functions.


Section~\ref{sec:setting} introduces Delone sets, their hull and
transversal, the groupoid $\GL$, and the standing conventions.
Section~\ref{sec:cocycle} constructs the address cocycle and establishes
the metric comparison. Section~\ref{sec:main} proves the grading theorem
and establishes symmetry and spectral invariance of
$\ell^{1,\nu}(\mathfrak C)$. Section~\ref{sec:kernels} develops the
kernel dictionary. Finally, Section~\ref{sec:consequences} derives
Theorem~A and its corollaries, proves Theorem~B and stretched-exponential
norm control, and discusses periodic and aperiodic obstructions at the
unweighted endpoint.

\section{Preliminaries}\label{sec:setting}

\subsection{Delone sets, hulls, transversals and the Delone groupoid}

\begin{definition}[Delone, FLC, repetitive]\label{def:delone}
Fix $d\ge1$ and $r,R>0$. A set $L\subset\R^d$ is called
\begin{enumerate}[label=\emph{(\roman*)},leftmargin=2.4em,itemsep=2pt]
\item \emph{$(r,R)$-Delone} if every open ball of radius $r$ meets $L$
in at most one point and every closed ball of radius $R$ meets $L$ in at
least one point;
\item of \emph{finite local complexity} (FLC) if for every $\rho>0$ the
collection of patches $\{(L-x)\cap B_\rho : x\in L\}$ is finite or equivalently, if $L-L$ is closed and discrete;
\item \emph{repetitive} if, for every $\rho>0$ and every centred patch
$P=(L-x)\cap\bar B_\rho$ with $x\in L$, the occurrence set
$\{y\in L:(L-y)\cap\bar B_\rho=P\}$ is relatively dense;
occurrence means exact patch agreement, including the absence of extra points;
\item \emph{aperiodic} if $L-t=L$ implies $t=0$ (under repetitivity this
is equivalent to aperiodicity of every element of the hull).
\end{enumerate}
\end{definition}
Throughout, $B_\rho(p)$ and $\bar B_\rho(p)$ denote the open and the
closed ball of radius $\rho$ centred at $p$, and
$B_\rho:=B_\rho(0)$, $\bar B_\rho:=\bar B_\rho(0)$.

\begin{definition}[Local matching topology, hull, transversal]\label{def:hull}
For a Delone set $L\subset\R^d$, $\rho>0$ and $\varepsilon\in(0,1)$ put
\[
U_{\rho,\varepsilon}(L):=\bigl\{L'\ \text{Delone}:\ \exists\, x,x'\in B_\varepsilon(0)\ \text{with}\
(L-x)\cap \bar B_\rho=(L'-x')\cap \bar B_\rho\bigr\}.
\]
The sets $U_{\rho,\varepsilon}(L)$ form a neighbourhood base of a Hausdorff
topology on the set of Delone sets, the \emph{local matching topology}
\cite[\S5.4]{BaakeGrimm13}, \cite[\S2]{LenzStollmann03}. (The topology
is defined through this base rather than through the customary formula
$d(L,L')=\inf\{\varepsilon:\ L'\in U_{1/\varepsilon,\varepsilon}(L)\}$, which
fails the triangle inequality and must be truncated to yield a metric; see
\cite[Rem.~5.6]{BaakeGrimm13}.) The \emph{continuous hull} is
\[
\hull=\hull(L):=\overline{\{L-t:\ t\in\R^d\}},
\]
the closure in this topology, and the \emph{canonical transversal} is
$\trans:=\{\omega\in\hull:\ 0\in\omega\}$. Translation $(t,\omega)\mapsto\omega-t$
defines an action of $\R^d$ on $\hull$, and $\trans$ is closed in $\hull$.
\end{definition}

\begin{proposition}[Standard facts about the hull]\label{prop:hullfacts}
Let $L$ be FLC. Then:
\begin{enumerate}[label=\emph{(H\arabic*)},leftmargin=2.8em,itemsep=2pt]
\item\label{H1} $\hull$ is compact and metrizable, and the $\R^d$-action
$(t,\omega)\mapsto\omega-t$ is jointly continuous;
\item\label{H2} $\trans$ is closed in $\hull$, hence compact, and totally
disconnected;
\item\label{H3} if $L$ is repetitive, the $\R^d$-action on $\hull$ is minimal;
\item\label{H4} if $L$ is repetitive and aperiodic, no $\omega\in\hull$
has a non-zero period, i.e.\ the $\R^d$-action on $\hull$ is free.
\end{enumerate}
\end{proposition}

\begin{proof}
\ref{H1} and \ref{H3} are classical
\cite[Thm.~3.1(i), Thm.~3.2]{LagariasPleasants03}, \cite[Prop.~5.4]{BaakeGrimm13};
in fact \ref{H3} is an equivalence. \ref{H2} is Lemma~\ref{lem:collapse}
below, whose proof uses only \ref{H1} and Lemma~\ref{lem:patch}.
For \ref{H4}: for fixed $t$ the set $\{\omega\in\hull:\omega-t=\omega\}$
is closed by \ref{H1}. If $\omega\in\hull$ had a period $t\neq0$, by
\ref{H3} there are $y_n$ with $\omega-y_n\to L$, each $\omega-y_n$ has
period $t$, so $L-t=L$, contradicting aperiodicity.
\end{proof}

\begin{remark}
(a) Under FLC the local matching topology agrees on $\hull$ with the Fell
topology on closed subsets of $\R^d$ \cite{BaakeLenz04}; this is not
needed here. (b) Repetitivity cannot be dropped in \ref{H4}:
$L=\Z\setminus\{0\}$ is aperiodic while $\Z\in\hull(L)$.
(c) That every $\omega\in\hull$ is again $(r,R)$-Delone is not cited;
it is Lemma~\ref{lem:patch}(i). That $\trans$ is a Cantor set under
repetitivity and aperiodicity is Corollary~\ref{cor:cantor}.
\end{remark}

\begin{definition}[Patches]\label{def:patch}
Let $L\subseteq\R^d$, $\rho>0$ and $x\in\R^d$. The \emph{$\rho$-patch of
$L$ at $x$} is the portion of $L$ seen in $\bar B_\rho(x)$, translated
back to the origin:
\[
P_\rho(L,x)\ :=\ (L-x)\cap\bar B_\rho\ \subseteq\ \bar B_\rho .
\]
Recording patches at the origin makes two of them equal precisely when
the corresponding portions of the point sets agree after a translation.
A set $P\subseteq\bar B_\rho$ \emph{occurs in $L$ at $y$} if
$P_\rho(L,y)=P$, and \emph{occurs in $L$} if it occurs at some
$y\in\R^d$; note that if $0\in P$ then necessarily $y\in L$. The
collection of such patches is denoted by
\[
\mathcal P_\rho\ =\ \mathcal P_\rho(L)\ :=\ \{P_\rho(L,x)\ :\ x\in L\}
\]
Each such patch contains $0$ and is finite as soon as $L$ is uniformly
discrete.
\end{definition}
\begin{remark}[Decorated Delone sets]\label{rem:decorated}
Fix a finite set $\mathcal A$ of \emph{decorations} (atomic species,
letters of a subshift, local data). A \emph{decorated point set} is a
pair $(L,\kappa)$ with $L\subset\R^d$ and $\kappa:L\to\mathcal A$;
translations act by $(L,\kappa)-t:=\big(L-t,\,\kappa(\cdot+t)\big)$,
and a \emph{decorated patch} is a restriction
$\big(L\cap\bar B_\rho,\,\kappa|_{L\cap\bar B_\rho}\big)$, translated
likewise. The decorated set is called Delone, respectively FLC, when
its underlying set $L$ has the corresponding property; both conditions
concern positions alone. Indeed, if $L$ is Delone its $\rho$-patches
are finite sets, and each of the finitely many
positional patches $P$ carries at most $|\mathcal A|^{\#P}$
decorations, and finiteness of decorated patches follows.
\emph{Repetitivity and aperiodicity, by contrast, are imposed on the
decorated object}: every decorated patch recurs within a uniform
radius, and $(L,\kappa)-t=(L,\kappa)$ forces $t=0$. Decorated
aperiodicity is a strictly weaker hypothesis than positional
aperiodicity, and this is the point of the notion: $L=\Z^d$ with an
aperiodic decoration is covered, which is how $\Z^d$-subshifts and
random binary alloys (Example~\ref{ex:random}) enter the framework.
The results extend verbatim to decorated sets when the local matching
topology in \S\ref{sec:setting} is defined using decorated patches. Two
patches match when, after a small translation, both their positions and
decorations agree. The hull and transversal then consist of decorated
configurations. Three observations justify this extension.

\emph{(1)} Lemmas~\ref{lem:patch} and~\ref{lem:collapse}, together with
\eqref{eq:domclopen}, remain valid as stated. Their proofs use
decorations only through the finiteness of $\mathcal P_\rho$. The metric
arguments, including the F{\o}lner sequence in the proof of
Proposition~\ref{prop:basic}, depend only on positions.

\emph{(2)} The group $\Gamma_L=\langle L-L\rangle$, the address map,
the cocycle $c(\cdot,x)=a(x)$, and the comparison
Lemma~\ref{lem:addcomp} also depend only on positions. Groupoid arrows
remain pairs consisting of a decorated configuration and a return
vector, so the weights and decay classes over $\Z^N$ are unchanged.

\emph{(3)} Decorations enter only through the unit space,
$\mathfrak C_0\cong C(\trans)$, where $\trans$ is the decorated
transversal. Decoration labels are not matrix indices, so the
coefficients $M_n(\Cx)$ in case (ii) remain available for internal
degrees of freedom such as orbitals and spin. Kernels may depend on the
configuration's decorations through local dependence condition (b).
This includes species-dependent hopping amplitudes, which form
continuous covariant families over the decorated hull.

Finally, decorations cannot be traded for vacancies: deleting the points of one color in general destroys the Delone property (a Bernoulli dilution of $\Z^d$ almost surely contains arbitrarily large empty balls), so for alloys the decorated framework is a necessity, not a convenience.

\end{remark}

\begin{definition}[Transversal groupoid]\label{def:groupoid}
Let $\trans$ be the transversal, i.e.\ the patterns $\omega$ with $0\in\omega$.
The \emph{transversal groupoid} is
\[
  \GL:=\bigl\{(\omega,x)\in\trans\times\R^d : x\in\omega\bigr\},
\]
with range, source, product, and inverse
\begin{align*}
  r(\omega,x)&=\omega, & s(\omega,x)&=\omega-x,\\
  (\omega,x)\cdot(\omega-x,y)&=(\omega,x+y), & (\omega,x)^{-1}&=(\omega-x,-x),
\end{align*}
the product being defined exactly when $s(\omega,x)=r(\omega-x,y)$.
The unit space is $\GL^{(0)}=\{(\omega,0):\omega\in\trans\}$, identified with
$\trans$ via $(\omega,0)\mapsto\omega$.
Also, the topology of this groupoid is generated by the basic sets
\[
U_{P,x}:=\bigl\{(\omega,x)\in\GL : \omega\cap \bar B_\rho=P\bigr\},
\]
for every $\rho\in(0,+\infty)$, $P\in\mathcal P_\rho$ and $x\in P$.
\end{definition}

\begin{remark}[Relation with Deaconu--Renault groupoids]\label{rem:DR}
In the following cases, transformation-groupoid and
partial-action descriptions are available.

\emph{(i) Dimension one.} Let $d=1$ (any rank $N$). For $\omega\in\trans$
put $x_1(\omega):=\min\big(\omega\cap(0,\infty)\big)$ and
$\sigma(\omega):=\omega-x_1(\omega)$. By uniform discreteness and
relative denseness, $r\le x_1(\omega)\le2R$, and $x_1$ is locally
constant on $\trans$ (Lemma~\ref{lem:collapse}), so $\sigma$ is a
homeomorphism of $\trans$ with inverse
$\omega\mapsto\omega-\max(\omega\cap(-\infty,0))$. Enumerating
$\omega=\{x_n(\omega)\}_{n\in\Z}$ with $x_0=0$, the map
$(\omega,x_n(\omega))\mapsto(\omega,n)$ is an isomorphism of \'etale
groupoids $\GL\cong\trans\rtimes_\sigma\Z$, the transformation
groupoid of $\sigma$ \cite{Deaconu95}. The counting cocycle
$(\omega,x_n)\mapsto n$ satisfies $r|n|\le|x_n|\le2R|n|$, which is
\eqref{eq:bilip} with $\Z$ in place of $\Z^N$; the address cocycle
$c(\omega,x_n)=a(x_n)=\sum_{k}a\big(x_{k+1}-x_k\big)$ is a
$\Z^N$-valued cocycle over $\sigma$.
In rank one, write $\Gamma_L=\ell\Z$ with $\ell>0$ and
choose the positive generator. The address and counting
cocycles agree exactly when every consecutive gap equals
$\ell$, equivalently when $L$ is a translate of $\ell\Z$.
Rank one alone does not imply this: for
$L=3\Z\cup(3\Z+1)$, one has $\Gamma_L=\Z$ but
$x_2=3$, so the address is $3$ while the counting index is $2$.
The groupoid isomorphism with a $\Z$-action therefore does not, by
itself, identify the counting-grading norm with the address-grading
common-profile norm: a fixed counting index can correspond to
several physical displacements across the hull.

\emph{(ii) Subsets of a lattice.} Let $L\subseteq\Gamma$ for a lattice
$\Gamma\subset\R^d$ (decorated lattices, vacancy models). Then
$\Gamma_L\subseteq\Gamma$ has finite index, $N=d$, and if
$\Gamma_L=\Z^d$ with the standard basis then $a=\mathrm{id}$ and
the comparison \eqref{eq:bilip} holds with $c_\pm=1$.
Choose $p\in L$ and replace $L$ by $L-p$.
This does not change its continuous hull or its return group,
and now $L-p\subseteq\Gamma_L$.
For the remainder of this paragraph, set
$\Gamma:=\Gamma_L$, so that the weight $\nu\circ a$ is
defined on the whole acting lattice.
The discrete hull
$\hull_\Gamma:=\overline{\{L-\gamma:\gamma\in\Gamma\}}\subseteq\{0,1\}^{\Gamma}$
carries the $\Gamma$-action by translation, $\trans=\{\omega\in\hull_\Gamma:0\in\omega\}$
is clopen in $\hull_\Gamma$ and meets every orbit, and
$\GL=(\hull_\Gamma\rtimes\Gamma)|_{\trans}$ is the reduction of the
transformation groupoid of that action to $\trans$;
equivalently, the transformation groupoid of the partial action of
$\Gamma$ on $\trans$ \cite{Abadie03,ExelBook}. Since
$\ell^{1,\nu}(\GL)=\chi_\trans\,\ell^{1,\nu}(\hull_\Gamma\rtimes\Gamma)\,\chi_\trans$
for the weight $\nu\circ a$ on $\Gamma$, and inverse-closedness passes
to corners, the \emph{unweighted} part of Theorem~\ref{thm:main}
also follows from the partial-action results of \cite{JM}.
For arbitrary GRS weights, one additionally uses
Proposition~\ref{prop:weightedSI}; the cited unweighted
partial-action theorem alone does not supply that weighted conclusion.

Related topological descriptions are available under the hypotheses
of \cite{BenedettiGambaudo03,SadunWilliams03}. In particular,
\cite[Thm.~2]{SadunWilliams03} treats tiling spaces built from
finitely many polyhedral prototiles meeting full-face to full-face,
counting distinct orientations as distinct tile types, and produces
a homeomorphism to a $d$-fold suspension of a $\Z^d$-subshift.
That statement is not a translation-equivariant conjugacy and does
not by itself identify the fixed physical profile norm used here.
No general Delone-to-suspension reduction is used; the needed metric
information is proved directly in Lemma~\ref{lem:addcomp}.
\end{remark}

\subsection{The topology of the transversal}
The next two lemmas isolate what this topology does on the transversal.
The first identifies bounded portions of hull elements as translates of $L$.
The second shows that small relative translations disappear on the
transversal, and that its topology is generated by exact agreement of patches.

\begin{lemma}[Patches]\label{lem:patch}
Let $L$ be an $(r,R)$-Delone set with FLC. Then:
\begin{enumerate}[label=\emph{(\roman*)},itemsep=2pt]
\item for every $\omega\in\hull$ and every $\rho>0$ there is
$v\in\R^d$ with
\[
\omega\cap\bar B_\rho=(L-v)\cap\bar B_\rho ;
\]
consequently every $\omega\in\hull$ is again $(r,R)$-Delone, and
$\omega-\omega\subseteq L-L$;
\item if moreover $\omega\in\trans$, then $v$ may be taken in $L$. Since
$0\in\omega$, this says that the $\rho$-patch of $\omega$ at the origin is
a $\rho$-patch of $L$; conversely $L-x\in\trans$ for every $x\in L$. Hence
\[
\{\,P_\rho(\omega,0)\ :\ \omega\in\trans\,\}\ =\ \mathcal P_\rho(L)
\qquad(\rho>0),
\]
a finite set by FLC. Moreover $\omega\subseteq L-L$.
\end{enumerate}
\end{lemma}

\begin{proof}
(i) Fix $\rho>0$ and choose $\eta\in(0,1)$. Since
$\omega\in\hull(L)$, there exist $t\in\R^d$ and
$u,u^{\prime}\in B_\eta$ such that
\[
(\omega-u)\cap\bar B_{\rho+\eta}
=
(L-t-u^{\prime})\cap\bar B_{\rho+\eta}.
\]
Translating both sides by $u$, the configurations $\omega$ and
$L-t-u^{\prime}+u$ agree on $\bar B_{\rho+\eta}+u$. Since
$\bar B_\rho\subseteq\bar B_{\rho+\eta}+u$, restriction to
$\bar B_\rho$ gives
\[
\omega\cap\bar B_\rho
=
(L-v)\cap\bar B_\rho,
\qquad v:=t+u^{\prime}-u.
\]

Uniform discreteness: if an open ball $B(p,r)$ met $\omega$ in two
distinct points, then taking $\rho\ge|p|+r$ the same ball would meet
the translate $L-v$ in two distinct points, contradicting
Definition~\ref{def:delone}. Relative denseness: given $p$, take
$\rho\ge|p|+R$; then $\bar B(p,R)\subseteq\bar B_\rho$ meets $L-v$,
and any such intersection point lies in $\omega$. Finally, if
$x,y\in\omega$ and $\rho\ge\max(|x|,|y|)$, then $x=\ell_1-v$ and
$y=\ell_2-v$ with $\ell_1,\ell_2\in L$, whence
$x-y=\ell_1-\ell_2\in L-L$.

(ii) If $0\in\omega$ then $0\in\omega\cap\bar B_\rho=(L-v)\cap\bar B_\rho$,
so $v\in L$ and $P_\rho(\omega,0)=\omega\cap\bar B_\rho=P_\rho(L,v)$, a
member of $\mathcal P_\rho(L)$. For the reverse inclusion, $L-x\in\trans$
for every $x\in L$, and $P_\rho(L-x,0)=P_\rho(L,x)$. Finiteness is FLC.
Finally, taking $\rho\ge|x|$ for $x\in\omega$ gives
$x\in L-v\subseteq L-L$.
\end{proof}

\begin{lemma}[No translation on the transversal]\label{lem:collapse}
Let $\omega,\omega'\in\trans$ and suppose that
\[
(\omega'+s)\cap\bar B_\rho\;=\;\omega\cap\bar B_\rho
\qquad\text{for some }|s|<\varepsilon ,
\]
with $0<\varepsilon\le\min\{r,\rho\}$. Then $s=0$. Consequently,
writing
\[
\trans_\rho(\omega):=\big\{\omega'\in\trans\ :\
\omega'\cap\bar B_\rho=\omega\cap\bar B_\rho\big\},
\]
the family $\{\trans_\rho(\omega)\}_{\rho>0}$ is a neighbourhood basis
at $\omega$ in $\trans$, each $\trans_\rho(\omega)$ is compact open, and
$\trans$ is a compact, metrizable, totally disconnected space.
\end{lemma}

\begin{proof}
First consider the displayed hypothesis in the statement. Since
$0\in\omega'$, it follows that $s\in\omega'+s$. The inequality
$|s|<\varepsilon\le\rho$ therefore gives
$s\in\omega\cap\bar B_\rho$. Both $0$ and $s$ belong to
$\omega\cap B_r$, so uniform discreteness implies $s=0$.

To compare the neighbourhood systems, fix $\rho>0$ and
$0<\varepsilon<\min\{1,r/2\}$. The following inclusions hold:
\begin{equation}\label{eq:transversal-neighbourhoods}
U_{\rho+\varepsilon,\varepsilon}(\omega)\cap\trans
\subseteq
\trans_\rho(\omega)
\subseteq
U_{\rho,\varepsilon}(\omega)\cap\trans .
\end{equation}
The second inclusion follows by taking both translations equal to
zero.

For the first inclusion, let $\omega'\in\trans$ and choose
$u,v\in B_\varepsilon$ such that
\[
(\omega-u)\cap\bar B_{\rho+\varepsilon}
=
(\omega'-v)\cap\bar B_{\rho+\varepsilon}.
\]
Since $0\in\omega$, the point $-u$ belongs to the left-hand side.
Consequently $v-u\in\omega'$. But $0\in\omega'$ and
$|v-u|<2\varepsilon<r$, so uniform discreteness gives $u=v$.
Undoing this common translation, the configurations agree on
$\bar B_{\rho+\varepsilon}+u$, which contains $\bar B_\rho$.
Thus $\omega'\in\trans_\rho(\omega)$.

The inclusions \eqref{eq:transversal-neighbourhoods} show that the
sets $\trans_\rho(\omega)$ form a neighbourhood basis. Each is
open: the same first inclusion applies at every configuration
belonging to that cylinder. For fixed $\rho$,
Lemma~\ref{lem:patch}(ii) gives a finite partition of $\trans$
into these open cylinders. Hence every cylinder is also closed.
Distinct configurations differ on some bounded ball, so this
clopen basis separates points and $\trans$ is totally disconnected.

For completeness, $\trans$ is closed in $\hull$. Indeed, let
$\omega_n\to\omega$ with $0\in\omega_n$. By local matching, for all
sufficiently large $n$ there exist $u_n,v_n\to0$ such that
\[
(\omega_n-u_n)\cap\bar B_1
=
(\omega-v_n)\cap\bar B_1 .
\]
Then $-u_n$ belongs to the left-hand side, and therefore
$v_n-u_n\in\omega$. These points converge to $0$; since $\omega$
is uniformly discrete and hence closed, $0\in\omega$.

Finally, $\hull$ is compact and metrizable by
Proposition~\ref{prop:hullfacts}\ref{H1}. Its closed subspace
$\trans$ is therefore compact and metrizable, and each clopen
cylinder is compact.
\end{proof}

\begin{proposition}[Basic structure]\label{prop:basic}
Let $L$ be FLC. Then $\GL$ is a Hausdorff, étale, ample, second
countable groupoid with compact unit space $\trans$, and it is
amenable. If $L$ is repetitive, $\GL$ is minimal. If $L$ is also
aperiodic, $\GL$ is principal, and consequently $\Cst_r(\GL)$ is simple.
\end{proposition}

\begin{proof}
For $x\in\R^d$ write $\dom_x:=\{\omega\in\trans:x\in\omega\}$; for
$P\in\mathcal P_\rho$ write $\trans_\rho(P)$ for the corresponding
member of the partition of Lemma~\ref{lem:collapse} and
$U_{P,x}:=\trans_\rho(P)\times\{x\}$ for $x\in P$, in accordance with
Definition~\ref{def:groupoid}.

\medskip\noindent\textbf{Step 1: the topology of $\GL$.}
For $\rho\ge|x|$ one has
\begin{equation}\label{eq:domclopen}
\dom_x=\bigsqcup\big\{\trans_\rho(P)\ :\ P\in\mathcal P_\rho,\ x\in P\big\},
\end{equation}
a finite union of clopen sets whose complement is the union of the
remaining members of the finite partition; hence \emph{$\dom_x$ is
clopen in $\trans$}.

By Lemma~\ref{lem:patch}(ii), $\GL\subseteq\trans\times(L-L)$, and
$L-L$ is closed and discrete by FLC. In the topology induced from
$\trans\times\R^d$ the second coordinate is therefore locally constant:
choosing $\delta>0$ with $B(x,\delta)\cap(L-L)=\{x\}$,
\[
\big(V\times B(x,\delta)\big)\cap\GL=(V\cap\dom_x)\times\{x\}
\qquad (V\subseteq\trans\text{ open}).
\]
Combining with Lemma~\ref{lem:collapse} and \eqref{eq:domclopen}, the
sets $U_{P,x}$ form a basis of this topology; that is, the cylinder
topology of Definition~\ref{def:groupoid} \emph{is} the topology induced
from $\trans\times\R^d$. In particular $\GL$ is Hausdorff. It is second
countable, a countable basis being $\{U_{P,x}:\rho\in\N,\
P\in\mathcal P_\rho,\ x\in P\}$, each $\mathcal P_\rho$ being finite and
each $P$ finite.

Continuity of the operations is now immediate: on $U_{P,x}$ the second
coordinate is constant, inversion is $(\omega,x)\mapsto(\omega-x,-x)$
with $\omega\mapsto\omega-x$ continuous by Proposition~\ref{prop:hullfacts}\ref{H1}, and multiplication
$\big((\omega,x),(\omega-x,y)\big)\mapsto(\omega,x+y)$ is the identity
in the first coordinate and constant in the second. The unit space is
$\{(\omega,0):\omega\in\trans\}$, identified with $\trans$; it is
compact by Lemma~\ref{lem:collapse}.

\medskip\noindent\textbf{Step 2: étale.}
Fix $\rho$, $P\in\mathcal P_\rho$ and $x\in P$, and note
$\trans_\rho(P)\subseteq\dom_x$. The map
$r|_{U_{P,x}}:(\omega,x)\mapsto\omega$ is a bijection onto
$\trans_\rho(P)$, continuous with continuous inverse
$\omega\mapsto(\omega,x)$, hence a homeomorphism onto a clopen subset of
$\trans$. Since $x\in\eta$ and $0\in\eta$ imply $0,-x\in\eta-x$, the
translation $\eta\mapsto\eta-x$ is a homeomorphism of $\dom_x$ onto
$\dom_{-x}$ by \ref{H1}, with inverse $\eta\mapsto\eta+x$; therefore
$s|_{U_{P,x}}:(\omega,x)\mapsto\omega-x$ is a homeomorphism of
$U_{P,x}$ onto $\trans_\rho(P)-x$, which is clopen in the clopen set
$\dom_{-x}$, hence clopen in $\trans$. Thus $r$ and $s$ are local
homeomorphisms and $\GL$ is étale.

\medskip\noindent\textbf{Step 3: ample.}
Each $U_{P,x}$ is a bisection because its fixed second coordinate means
that each arrow is determined by its range and by its source. It is
compact, as the range map identifies it with the clopen subset
$\trans_\rho(P)$ of the compact space $\trans$. By Step~1, these sets
form a basis of compact open bisections for $\GL$. Thus $\GL$ is ample,
or equivalently, étale with totally disconnected unit space
(Lemma~\ref{lem:collapse}).

\medskip\noindent\textbf{Step 4: amenability.}
The \emph{topological approximate invariant density} criterion for
amenability of an étale groupoid is verified here
\cite[Def.~2.2.8 and Prop.~2.2.13]{ADR}; equivalently, the continuous-probability-mean definition in
\cite[Def.~4.1.2]{Sims18}. It suffices to construct a sequence
$(g_n)\subset C_\co(\GL)$, $g_n\ge0$, with
$\sum_{\eta\in\GL^{u}}g_n(\eta)\le1$ for every $u\in\trans$ and with
\begin{equation}\label{eq:folner}
\sum_{\eta\in\GL^{u}}g_n(\eta)\longrightarrow1
\ \ \text{uniformly in }u\in\trans,
\qquad
\sum_{\eta\in\GL^{r(\gamma)}}\big|g_n(\gamma^{-1}\eta)-g_n(\eta)\big|
\longrightarrow0
\end{equation}
uniformly on compact subsets of $\GL$, where $\GL^{u}:=r^{-1}(u)$
(the unit space being compact, uniformity on compacta of $\trans$
is uniformity on $\trans$; the functions $g_n$ below satisfy the first sum
$=1$ exactly, so the boundedness clause holds automatically).
Indeed, $g_n|_{\GL^u}$ then gives a probability measure on each
range fibre; continuity of these systems follows from the continuity
of $g_n$ and the counting Haar system.

Let $\varphi_n(z):=\max\{0,\min\{1,\,n+1-|z|\}\}$, so that $\varphi_n$
is $1$-Lipschitz with values in $[0,1]$, equals $1$ on $\bar B_n$ and
vanishes on $\{|z|\ge n+1\}$. Put
\[
m_n(\omega):=\sum_{z\in\omega}\varphi_n(z),
\qquad
g_n(\omega,x):=\frac{\varphi_n(x)}{m_n(\omega)} .
\]
Each sum is finite by Lemma~\ref{lem:patch}(i), and
$m_n(\omega)\ge\varphi_n(0)=1$, so $g_n$ is well defined. Since
$m_n(\omega)$ depends only on $\omega\cap\bar B_{n+1}$, it is locally
constant on $\trans$ by Lemma~\ref{lem:collapse}; hence $g_n$ is
continuous. Its support is contained in
$\bigcup\{\dom_x\times\{x\}:x\in P,\ P\in\mathcal P_{n+1}\}$, a finite
union of compact sets, so $g_n\in C_\co(\GL)$ and $g_n\ge0$.

Since $\GL^{\omega}=\{(\omega,x):x\in\omega\}$, the first condition in
\eqref{eq:folner} holds \emph{exactly}:
$\sum_{\eta\in\GL^{\omega}}g_n(\eta)
=m_n(\omega)^{-1}\sum_{x\in\omega}\varphi_n(x)=1$ for every $n$ and
every $\omega$.

For the second, fix $\gamma=(\omega,y)\in\GL$. For
$\eta=(\omega,x)\in\GL^{r(\gamma)}=\GL^{\omega}$ one has
$\gamma^{-1}=(\omega-y,-y)$ and $\gamma^{-1}\eta=(\omega-y,\,x-y)$,
while $m_n(\omega-y)=\sum_{x\in\omega}\varphi_n(x-y)$. Writing
$f(x):=\varphi_n(x-y)$ and $h(x):=\varphi_n(x)$ as functions on
$\omega$, with $F:=\|f\|_{\ell^1(\omega)}$ and
$H:=\|h\|_{\ell^1(\omega)}$, the identity
$\frac fF-\frac hH=\frac{f-h}{H}+\frac{f\,(H-F)}{FH}$ together with
$|H-F|\le\|f-h\|_1$ gives
\[
\sum_{\eta\in\GL^{r(\gamma)}}\big|g_n(\gamma^{-1}\eta)-g_n(\eta)\big|
=\Big\|\frac fF-\frac hH\Big\|_{\ell^1(\omega)}
\ \le\ \frac{2\,\|f-h\|_{\ell^1(\omega)}}{m_n(\omega)} .
\]

\emph{Denominator (relative denseness).} For $n\ge2R$, the balls
$\bar B(x,R)$ with $x\in\omega\cap\bar B_n$ cover
$\bar B_{n-R}$. Volume comparison therefore gives
\[
m_n(\omega)\ge\#(\omega\cap\bar B_n)
\ge\left(\frac{n-R}{R}\right)^d
\ge\left(\frac n{2R}\right)^d.
\]
This lower bound is uniform in $\omega$.

\emph{Numerator (uniform discreteness).} The summand
$|\varphi_n(x-y)-\varphi_n(x)|$ is at most $\min(1,|y|)$ and vanishes
unless $n-|y|<|x|<n+1+|y|$: below that range both values equal $1$,
above it both vanish. The points of $\omega$ in that annulus are
$2r$-separated, so the balls $B(x,r)$ around them are disjoint and
contained in a concentric annulus of thickness $1+2|y|+2r$; comparing
volumes and using $b^d-a^d\le d\,b^{d-1}(b-a)$,
\[
\|f-h\|_{\ell^1(\omega)}\ \le\
d\,r^{-d}\,(n+1+|y|+r)^{d-1}\,(1+2|y|+2r).
\]

A compact subset of $\GL$ meets only finitely many slices
$\dom_x\times\{x\}$ because its image under the continuous projection
to the discrete set $L-L$ is compact and therefore finite. Its second
coordinates are consequently bounded, so it lies in
$\{|y|\le\Lambda\}$ for some $\Lambda$. Combining the two bounds,
\[
\sum_{\eta\in\GL^{r(\gamma)}}\big|g_n(\gamma^{-1}\eta)-g_n(\eta)\big|
\ \le\ C(d,r,R,\Lambda)\,\frac1n
\qquad\text{whenever }|y|\le\Lambda ,
\]
uniformly in $\omega$. Thus \eqref{eq:folner} holds and $\GL$ is
topologically amenable; being second countable, Hausdorff and étale, it
is amenable, and $\Cst(\GL)=\Cst_r(\GL)$ is nuclear
\cite[Prop.~6.1.8, Cor.~6.2.14]{ADR}; see also
\cite[Thms.~4.1.4--4.1.5]{Sims18}. The general theory of groupoid
$\Cst$-algebras used here is that of \cite{Renault}.

\medskip\noindent\textbf{Step 5: minimality.}
Assume $L$ repetitive. Since $0\in\omega-t$ iff $t\in\omega$, the
$\GL$-orbit of $\omega\in\trans$ is
$\{\omega-x:x\in\omega\}=(\R^d\!\cdot\omega)\cap\trans$. Let
$\eta\in\trans$ and $\rho>0$; it remains to find $x\in\omega$ with
$\omega-x\in\trans_\rho(\eta)$. By Proposition~\ref{prop:hullfacts}\ref{H3} the orbit $\R^d\!\cdot\omega$ is
dense in $\hull$.
Using the enlarged-ball argument from the proof of
Lemma~\ref{lem:patch}, there exists $x\in\R^d$ such that
\[
(\omega-x)\cap\bar B_\rho=\eta\cap\bar B_\rho.
\]
Since $0\in\eta$, this equality implies $x\in\omega$.
Thus $\omega-x\in\trans_\rho(\eta)$, proving minimality.

\medskip\noindent\textbf{Step 6: principality and simplicity.}
An element $(\omega,x)$ of the isotropy satisfies
$\omega=r(\omega,x)=s(\omega,x)=\omega-x$, i.e.\ $x$ is a period of
$\omega$; by Proposition~\ref{prop:hullfacts}\ref{H4} the action is free, so $x=0$ and $\GL$ is principal,
hence topologically principal. Being second countable, locally compact,
Hausdorff and étale, minimal and topologically principal, $\GL$ has
simple $\Cst_r(\GL)$ \cite{BCFS}; by Step~4 and \cite[Prop.~6.1.8]{ADR} the full and
reduced algebras coincide, so $\Cst(\GL)$ is simple as well.
\end{proof}

\begin{corollary}\label{cor:cantor}
If $L$ is repetitive and aperiodic then $\trans$ has no isolated points and it is a Cantor set (compact, metrizable and totally disconnected).
\end{corollary}

\begin{proof}
Suppose $\{\omega_0\}$ were open in $\trans$. By Step~5 the orbit of
every $\eta\in\trans$ is dense, hence meets $\{\omega_0\}$; as lying in
a common orbit is an equivalence relation, $\trans$ is the single orbit
$\{\omega_0-x:x\in\omega_0\}$. For $x\in\omega_0$ the map
$\eta\mapsto\eta-x$ is a homeomorphism of $\dom_x$ onto $\dom_{-x}$
(Step~2) between clopen subsets of $\trans$, so $\{\omega_0-x\}$ is open
and every point of $\trans$ is isolated; a compact space all of whose
points are isolated is finite, so $\trans$ is finite. But by \ref{H4} the
map $x\mapsto\omega_0-x$ is injective on $\omega_0$, so $\omega_0$ would
be finite, contradicting relative denseness. Now apply
Lemma~\ref{lem:collapse} and Brouwer's characterization of the Cantor set.
\end{proof}

\subsection{The address cocycle}\label{sec:cocycle}

The address-map geometry of \cite{Lagarias}, including the explicit
constants and hull compatibility used in the weighted estimates, is
recalled here.
\begin{proposition}[{Lagarias \cite{Lagarias}}]\label{prop:lagarias}
Let $L\subset\R^d$ be an FLC Delone set. Then
$\Gamma_L:=\langle L-L\rangle\le(\R^d,+)$ is free abelian of finite rank
$N\ge d$. Any choice of basis gives the \emph{address map}, a group
isomorphism $a:\Gamma_L\to\Z^N$.
\end{proposition}

\begin{proof}
By FLC the set $F:=(L-L)\cap\bar B_{3R}$ is finite. It suffices to
consider $z=y_1-y_0\in(L-L)\setminus\{0\}$. Put $m:=\lceil|z|/R\rceil$, $p_j:=y_0+\tfrac jm z$,
and choose $x_j\in L\cap\bar B_R(p_j)$ for $0<j<m$, with $x_0:=y_0$,
$x_m:=y_1$; then $x_{j+1}-x_j\in F$ for every $j$ and
$z=\sum_j(x_{j+1}-x_j)$. Hence $\Gamma_L$ is generated by $F$. A
finitely generated torsion-free abelian group is free of finite rank,
and $N\ge d$ because $L-L$ contains $d$ linearly independent vectors
($L$ is relatively dense, so it is not contained in a proper affine
subspace). Lagarias works with $\langle L\rangle$ rather than
$\langle L-L\rangle$; the two agree when $0\in L$, and $\Gamma_L$ has
the advantage of being translation-invariant.
\end{proof}

\begin{lemma}[Hull compatibility]\label{lem:hullcompat}
Let $L$ be FLC. Then $\omega-\omega\subset\Gamma_L$ for every $\omega\in\hull(L)$; in fact every return vector of every $\omega$ in the hull is a return vector of $L$: $\bigcup_{\omega\in\hull}(\omega-\omega)\subset \overline{L-L}=L-L$.
\end{lemma}

\begin{proof}
Contained in Lemma~\ref{lem:patch}(i), whose final clause gives
$\omega-\omega\subseteq L-L\subseteq\langle L-L\rangle=\Gamma_L$ for
every $\omega\in\hull$, with no limit argument; and $L-L$ is closed by
FLC (Definition~\ref{def:delone}), so $\overline{L-L}=L-L$.
\end{proof}

\begin{proposition}[The address cocycle]\label{prop:cocycle}
Let $L$ be an FLC Delone set. Define
\[
c:\GL\longrightarrow\Z^N,\qquad c(\omega,x):=a(x).
\]
Then: (i) $c$ is a continuous (indeed locally constant) groupoid
cocycle; (ii) $c^{-1}(0)=\trans$ (the cocycle is faithful); (iii) for
every $g\in\Z^N$, letting $x_g:=a^{-1}(g)\in\Gamma_L$,
\[
c^{-1}(g)=\{(\omega,x_g):\ \omega\in\trans,\ x_g\in\omega\}
\]
is a \emph{compact open bisection} of $\GL$, with
$\dom_g:=\{\omega\in\trans: x_g\in\omega\}$ clopen in $\trans$.
Moreover $x_{-g}=-x_g$,
\[
r\big(c^{-1}(g)\big)=\dom_g,\qquad
s\big(c^{-1}(g)\big)=\dom_g-x_g=\dom_{-g},
\]
and $\dom_g\neq\varnothing$ if and only if $x_g\in L-L$; consequently
$c(\GL)=a(L-L)$.
\end{proposition}

\begin{proof}
First, $c$ is well defined: if $(\omega,x)\in\GL$ then
$x\in\omega\subseteq L-L\subseteq\Gamma_L$ by
Lemma~\ref{lem:patch}(ii), so $a(x)\in\Z^N$ makes sense.

(i) Let $\big((\omega,x),(\omega-x,y)\big)$ be a composable pair, with
product $(\omega,x+y)$ (Definition~\ref{def:groupoid}). As above
$x,y\in\Gamma_L$, hence $x+y\in\Gamma_L$, and additivity of the
isomorphism $a$ (Proposition~\ref{prop:lagarias}) gives
\[
c\big((\omega,x)(\omega-x,y)\big)=a(x+y)=a(x)+a(y)
=c(\omega,x)+c(\omega-x,y);
\]
in particular $c$ vanishes on units and $c(\gamma^{-1})=-c(\gamma)$.
Local constancy: the sets $U_{P,x}$ form a basis of the topology of
$\GL$ (Step~1 of the proof of Proposition~\ref{prop:basic}) and have
constant second coordinate, on which alone $c(\omega,x)=a(x)$ depends.
Since $\Z^N$ is discrete, $c$ is continuous and every $c^{-1}(g)$ is
clopen in $\GL$.

(ii) By injectivity of $a$, $c(\omega,x)=0$ iff $x=0$; thus
$c^{-1}(0)=\{(\omega,0):\omega\in\trans\}$ is the unit space,
identified with $\trans$ as in Definition~\ref{def:groupoid}.

(iii) Fix $g\in\Z^N$. If $(\omega,x)\in c^{-1}(g)$ then $x\in\Gamma_L$
with $a(x)=g$, so $x=x_g$ by injectivity; and $(\omega,x_g)\in\GL$
precisely when $x_g\in\omega$, i.e.\ $\omega\in\dom_g=\dom_{x_g}$,
clopen in $\trans$ by \eqref{eq:domclopen}. This proves the displayed
description of $c^{-1}(g)$, which is clopen in $\GL$ by (i).

The restriction $r|_{c^{-1}(g)}:(\omega,x_g)\mapsto\omega$ is a
bijection onto $\dom_g$ with inverse $\omega\mapsto(\omega,x_g)$.
Both maps are continuous because the topology on $\GL$ is induced from
$\trans\times\R^d$ (Step~1 of the proof of Proposition~\ref{prop:basic}).
Therefore, $r$ is a homeomorphism from $c^{-1}(g)$ onto the clopen set
$\dom_g$. In particular, $c^{-1}(g)$
is compact, $\dom_g$ being closed in the compact space $\trans$
(Lemma~\ref{lem:collapse}). Since $a$ is an isomorphism,
$x_{-g}=-x_g$; and $s|_{c^{-1}(g)}$ is the composition of
$r|_{c^{-1}(g)}$ with the translation $\omega\mapsto\omega-x_g$, a
homeomorphism of $\dom_{x_g}$ onto $\dom_{-x_g}=\dom_{-g}$ (Step~2 of
the proof of Proposition~\ref{prop:basic}); hence $s$ restricts to a
homeomorphism of $c^{-1}(g)$ onto the clopen set $\dom_{-g}$. Thus
$c^{-1}(g)$ is a compact open bisection: the unique arrow of degree
$g$ with range $\omega\in\dom_g$ is $(\omega,x_g)$, and the unique one
with source $\eta\in\dom_{-g}$ is $(\eta+x_g,\,x_g)$.

Finally, if $\omega\in\dom_g$ then
$x_g\in\omega\subseteq L-L$ (Lemma~\ref{lem:patch}(ii)); conversely,
if $x_g=p-q$ with $p,q\in L$, then $\omega:=L-q\in\trans$ and
$x_g\in\omega$. Hence $\dom_g\neq\varnothing\iff x_g\in L-L$, and
$c(\GL)=\{g:\dom_g\neq\varnothing\}=a(L-L)$.
\end{proof}

\begin{example}
For $L=\Z^d$ with a decoration (subshifts, Remark~\ref{rem:decorated}): $N=d$ and $c(\omega,x)=x$. Fibonacci: $N=2$. Penrose and Ammann--Beenker vertices: $N=4$. The icosahedral quasicrystal: $N=6$. Disordered binary alloys on $\Z^d$ (the hull of a Bernoulli configuration): $N=d$, see Example~\ref{ex:random}.
\end{example}

The address map trades the physical (Euclidean) metric of $\R^d$ for the
algebraic metric of $\Z^N$, and the kernel classes of \S\ref{sec:kernels}
are defined through the latter. The next lemma shows that, precisely
because of FLC, nothing is lost in the trade: on return vectors the two
metrics are bi-Lipschitz comparable, with constants uniform over the
hull. Fix a basis $v_1,\dots,v_N$ of $\Gamma_L$, let
$a:\Gamma_L\to\Z^N$ be the associated address map
(Prop.~\ref{prop:lagarias}), and let $\Psi:\R^N\to\R^d$ be the
(surjective) linear map with $\Psi(e_i)=v_i$, so that
$\Psi\circ a=\mathrm{id}_{\Gamma_L}$. Euclidean norms on $\R^d$ and
$\R^N$ are both written $|\cdot|$; any other choice of norms only
affects the constants below.

\begin{lemma}[Address--Euclidean comparison]\label{lem:addcomp}
Let $L\subset\R^d$ be an $(r,R)$-Delone set with FLC. The set
$(L-L)\cap\bar B_{3R}$ is finite; put
\[
M:=\max\big\{\,|a(z)|\ :\ z\in(L-L)\cap\bar B_{3R}\,\big\},\qquad
c_-:=\|\Psi\|^{-1},\qquad c_+:=M\Big(\frac1r+\frac1R\Big).
\]
Then
\begin{equation}\label{eq:bilip}
c_-\,|z|\ \le\ |a(z)|\ \le\ c_+\,|z|
\qquad\text{for all } z\in L-L ,
\end{equation}
and consequently, by Lemma~\ref{lem:hullcompat},
\[
c_-\,|x|\ \le\ |c(\omega,x)|\ \le\ c_+\,|x|
\qquad\text{for all }(\omega,x)\in\GL ,
\]
with the same constants for every configuration in the hull.
\end{lemma}

\begin{proof}
Finiteness of $(L-L)\cap\bar B_{3R}$ follows since $L-L$ is closed and
discrete (FLC) and $\bar B_{3R}$ is compact.

\emph{Lower bound.} For every $z\in\Gamma_L$ one has $z=\Psi(a(z))$,
hence $|z|\le\|\Psi\|\,|a(z)|$. (Only finite generation is used; this
inequality holds on all of $\Gamma_L$.)

\emph{Upper bound.} Let $z=y_1-y_0\neq0$ with $y_0,y_1\in L$. Put
$m:=\lceil|z|/R\rceil$ and $p_j:=y_0+\tfrac jm\,z$ for $0\le j\le m$.
By relative denseness choose $x_j\in L\cap\bar B(p_j,R)$ for $0<j<m$,
and set $x_0:=y_0$, $x_m:=y_1$. Each increment $z_j:=x_{j+1}-x_j$
belongs to $L-L$ and
\[
|z_j|\ \le\ |x_{j+1}-p_{j+1}|+|p_{j+1}-p_j|+|p_j-x_j|
\ \le\ R+\frac{|z|}{m}+R\ \le\ 3R ,
\]
so $|a(z_j)|\le M$. Since $a$ is additive and
$\sum_{j=0}^{m-1}z_j=z$,
\[
|a(z)|\ \le\ \sum_{j=0}^{m-1}|a(z_j)|\ \le\ mM
\ \le\ M\Big(\frac{|z|}{R}+1\Big)
\ \le\ M\Big(\frac1R+\frac1r\Big)|z| ,
\]
using $|z|\ge r$ in the last step (uniform discreteness).

\emph{Hull.} Every return vector of every $\omega\in\hull$ is a return
vector of $L$ (Lemma~\ref{lem:hullcompat}), so \eqref{eq:bilip} applies
verbatim; for $(\omega,x)\in\GL$ one has $x\in\omega-\omega$ because
$0\in\omega$, and $c(\omega,x)=a(x)$.
\end{proof}

\begin{remark}[FLC is exactly the comparability condition]\label{rem:flcsharp}
(i) The upper bound in \eqref{eq:bilip} \emph{characterizes} FLC:
Lagarias \cite{Lagarias} proves that a finitely generated Delone set
satisfies a Lipschitz bound $|a(z)|\le C|z|$ on $L-L$ if and only if it
is of finite type, i.e.\ FLC. The hypothesis of
Lemma~\ref{lem:addcomp} is therefore necessary as well as sufficient,
and the standing FLC assumption of this paper is precisely the
condition under which the algebraic and the physical decay scales
agree.
(ii) No such comparison can hold on all of $\Gamma_L$ when $N>d$: a
finitely generated subgroup of $\R^d$ of rank $N>d$ is non-discrete,
hence contains nonzero elements of arbitrarily small Euclidean norm,
while their addresses satisfy $|a(x)|\ge1$. The content of
\eqref{eq:bilip} is that the \emph{return vectors} are protected from
this collapse.
\end{remark}

\begin{example}[The collapse is real]\label{ex:fibcollapse}
Fibonacci ($d=1$, $N=2$): $\Gamma_L=\Z+\tau\Z$ with $\tau$ the golden
mean, basis $v_1=1$, $v_2=\tau$, so $c_-=(1+\tau^2)^{-1/2}$. The
elements $x_k:=F_{k+1}-F_k\,\tau\in\Gamma_L$ satisfy
$|x_k|=\tau^{-k}\to0$ while $|a(x_k)|=|(F_{k+1},-F_k)|\sim\tau^{k}$:
on $\Gamma_L$ the ratio $|a(x)|/|x|$ is unbounded. For the standard
normalization (tile lengths $1$ and $\tau$) none of the $x_k$, $k\ge1$,
is a return vector.
\end{example}

\begin{remark}[Effective constants]\label{rem:constants}
The constants in Lemma~\ref{lem:addcomp} are effective once a basis of
$\Gamma_L$, the Delone parameters, and a complete list of the short
return vectors $(L-L)\cap\bar B_{3R}$ are supplied:
$c_-^{-1}=\|\Psi\|$ is the operator norm of the $d\times N$ matrix
$[\,v_1|\cdots|v_N\,]$, and $M$ is a maximum over that finite list.
FLC guarantees its finiteness; it does not, by itself, give an algorithm
for extracting a certified complete list from an unspecified infinite
point set. The quantitative estimates also take the counting constants
$(\kappa,C_\kappa)$ as data. They depend on the
choice of basis of $\Gamma_L$, but a change of basis by
$A\in\mathrm{GL}_N(\Z)$ modifies them in a controlled way
($a\mapsto Aa$). These constants enter the quantitative bounds in
Theorems~\ref{thm:nci} and~\ref{thm:stretchednci}; no optimality
claim is made for the resulting control functions.
\end{remark}

\section{Main results}\label{sec:main}
\subsection{Coefficients and graded algebras}\label{sec:coeff}

\begin{convention}[Coefficient scope]\label{conv:coeffscope}
Throughout the untwisted results, an \emph{admissible coefficient bundle}
means exactly the trivial bundle
$\Fell=\GL\times M_n(\Cx)$ for a fixed finite integer $n\ge1$.
Case (i) is $n=1$ and case (ii) is general finite $n$. The operator
norm on $M_n(\Cx)$ is used, and the ambient algebra is
\[
\mathfrak C=C_r^*(\GL;M_n(\Cx))\cong M_n(C_r^*(\GL)).
\]
Statements written with scalar notation are read with this matrix
amplification when $n>1$. No theorem below quantifies over arbitrary
Fell bundles or infinite-dimensional coefficient fibres.
The specific scalar magnetic twist is treated separately in
Remark~\ref{rem:magnetic}.
\end{convention}

On $C_\co(\GL|\Fell)$ one has Renault's $I$-norm, the exact analogue of the Hahn norm $\|\cdot\|_{\infty,1}$ of \cite[\S2.1]{FJM}:
\[
\|\Phi\|_{I}:=\max\Big\{\ \sup_{\omega\in\trans}\sum_{x\in\omega}\|\Phi(\omega,x)\|\ ,\ \sup_{\omega\in\trans}\sum_{x\in\omega}\|\Phi(\omega-x,-x)\|\ \Big\}.
\]
Recall from \cite[Def.~2.1]{FJM} the notions of \emph{symmetric}, \emph{rigidly symmetric} and \emph{hypersymmetric} discrete groupoid, and from \cite[Def.~2.1--2.2]{JM} the notion of a \emph{topologically $\mathsf{G}$-graded $\Cst$-algebra} $\mathfrak{C}=\widetilde{\bigoplus}_{g}\mathfrak{C}_g$ and its $\ell^1$-algebra $\ell^1(\mathfrak{C})=\{\Phi:\sum_g\|\widetilde P_g(\Phi)\|<\infty\}$.

\subsection{The grading}

\begin{theorem}[Cocycle grading]\label{thm:grading}
Let $L$ be FLC and let $\Fell$ be an admissible Fell bundle over $\GL$ (\S\ref{sec:coeff}). Put $\mathfrak{C}:=\Cst_r(\GL|\Fell)=\Cst(\GL|\Fell)$ (full $=$ reduced by amenability, \cite[Prop.~6.1.8]{ADR}) and, for $g\in\Z^N$,
\[
\mathfrak{C}_g:=\overline{C_\co\big(c^{-1}(g)\,\big|\,\Fell\big)}\ \subset\ \mathfrak{C}.
\]
Then $\mathfrak{C}=\widetilde{\bigoplus}_{g\in\Z^N}\mathfrak{C}_g$ is a unital, topologically $\Z^N$-graded $\Cst$-algebra in the sense of \cite[Def.\ 2.1]{JM}, with
$\mathfrak{C}_0\cong C(\trans)$ (case (i)) resp.\ $C(\trans)\otimes M_n$ (case (ii)), and the conditional expectation $\widetilde P_0$ is the restriction of kernels to the unit space; it is faithful.
\end{theorem}

\begin{proof}
A complete, self-contained proof for case (i) is given; the amplification to case (ii) is indicated at the end. Throughout, $\|\cdot\|_r$ is the reduced norm; for $u\in\trans$, $\pi_u$ denotes the regular representation on $\ell^2(\mathcal{G}_u)$, where $\mathcal{G}_u:=s^{-1}(u)$, given by $(\pi_u(f)\xi)(\gamma)=\sum_{\alpha\beta=\gamma}f(\alpha)\xi(\beta)$. Proposition~\ref{prop:cocycle} is used freely: $c$ is a locally constant cocycle with $c^{-1}(0)=\trans$.

\medskip\noindent\textbf{Step 1: partition into compact open bisections.}
The sets $c^{-1}(g)$, $g\in\Z^N$, form a partition of $\GL$ into clopen subsets, and $c^{-1}(g)=\{(\omega,x_g):\omega\in\dom_g\}$ with $x_g:=a^{-1}(g)$ and $\dom_g=\{\omega\in\trans: x_g\in\omega\}$ clopen, hence compact, in $\trans$ (Proposition~\ref{prop:cocycle}(iii), \eqref{eq:domclopen}). Each $c^{-1}(g)$ is a bisection: an arrow in it is determined by its range (the vector $x_g$ being fixed), and likewise by its source; via $r$ it is homeomorphic to $\dom_g$.

\medskip\noindent\textbf{Step 2: the $\Cst$-norm on a bisection is the sup norm.}
First, for $h\in C_\co(\trans)$ each $\pi_u(h)$ is the diagonal operator $(\pi_u(h)\xi)(\gamma)=h(r(\gamma))\xi(\gamma)$; since $u\in\mathcal{G}_u$, taking suprema gives $\|h\|_r=\|h\|_\infty$. Next, let $U\subset\GL$ be an open bisection and $f\in C_\co(U)$. Then $U^{-1}U\subset\trans$ (two arrows of $U$ with the same range coincide), and
\[
(f^*\!*f)(u)=\sum_{s(\beta)=u,\ \beta\in U}|f(\beta)|^2=|f(\gamma_u)|^2 ,
\]
$\gamma_u$ being the unique element of $U\cap s^{-1}(u)$ (the term is $0$ if there is none). Hence $f^*\!*f\in C_\co(\trans)$ and
\[
\|f\|_r^2=\|f^*\!*f\|_r=\|f^*\!*f\|_\infty=\|f\|_\infty^2 .
\]
Consequences: $\mathfrak{C}_g=C_\co(c^{-1}(g))\cong C(\dom_g)$ \emph{isometrically} (already complete, no closure needed); the indicator $u_g:=\mathbf 1_{c^{-1}(g)}\in C_\co(\GL)$ is a partial isometry whose initial and final projections are indicators of clopen subsets of $\trans$; and $u_0=\mathbf 1_\trans$ is a unit for $\mathfrak{C}$ because $\trans$ is compact, so $\mathfrak{C}$ is unital with $\mathfrak{C}_0=C(\trans)$.

\medskip\noindent\textbf{Step 3: grading axioms and density.}
For $f\in C_\co(c^{-1}(g))$, $k\in C_\co(c^{-1}(h))$, every term of $(f*k)(\gamma)=\sum_{\alpha\beta=\gamma}f(\alpha)k(\beta)$ forces $c(\gamma)=c(\alpha)+c(\beta)=g+h$; thus $\mathfrak{C}_g\mathfrak{C}_h\subset\mathfrak{C}_{g+h}$, and $f^*(\gamma)=\overline{f(\gamma^{-1})}$ gives $\mathfrak{C}_g^*=\mathfrak{C}_{-g}$. Density: the compact support of any $f\in C_\co(\GL)$ meets only finitely many of the disjoint open sets $c^{-1}(g)$; since these are clopen, $f=\sum_g f\cdot\mathbf 1_{c^{-1}(g)}$ is a \emph{finite} sum of continuous sections. Hence $C_\co(\GL)=\bigoplus_g^{\mathrm{alg}}C_\co(c^{-1}(g))$ is dense in $\mathfrak{C}$.

\medskip\noindent\textbf{Step 4: gauge action and the projections (the heart).}
Define $\alpha:\T^N\to\mathrm{Aut}(\mathfrak{C})$ on $C_\co$ by
\[
(\alpha_z f)(\gamma):=z^{\,c(\gamma)}f(\gamma),\qquad z^{\,g}:=z_1^{g_1}\cdots z_N^{g_N}.
\]
The cocycle identity makes $\alpha_z$ multiplicative, and $c(\gamma^{-1})=-c(\gamma)$ makes it $*$-preserving. It is isometric for $\|\cdot\|_r$ because it is \emph{unitarily implemented} in every regular representation: with $(U_z\xi)(\gamma):=z^{\,c(\gamma)}\xi(\gamma)$ on $\ell^2(\mathcal{G}_u)$ one checks $U_z\,\pi_u(f)\,U_z^*=\pi_u(\alpha_z f)$, using that the matrix coefficient of $\pi_u(f)$ at $(\beta,\alpha)$ is $f(\beta\alpha^{-1})$ and $c(\beta\alpha^{-1})=c(\beta)-c(\alpha)$. Strong continuity: on $f\in C_\co$, $\alpha_z f-f=\sum_{g\in F}(z^g-1)f_g$ is a finite sum tending to $0$; density plus equicontinuity extend this to $\mathfrak{C}$. Now set, as a Bochner integral over normalized Haar measure,
\[
\widetilde P_g(a):=\int_{\T^N}\bar z^{\,g}\,\alpha_z(a)\dd z .
\]
Each $\widetilde P_g$ is contractive (average of isometries). On $f\in C_\co$ the integral is a finite computation and character orthogonality gives $\widetilde P_g(f)=f\cdot\mathbf 1_{c^{-1}(g)}$; by density, $\widetilde P_g$ maps $\mathfrak{C}$ onto $\mathfrak{C}_g$, acts as the identity on $\mathfrak{C}_g$ and as $0$ on $\mathfrak{C}_h$, $h\neq g$ (in particular the algebraic sum is direct). Thus $\widetilde P_0:\mathfrak{C}\to C(\trans)$ is a unital contractive projection onto a $\Cst$-subalgebra, hence by Tomiyama a positive conditional expectation, and all the axioms of \cite[Def.\ 2.1--2.2]{JM} are verified: the grading is topological.

\medskip\noindent\textbf{Step 5: faithfulness (here ``reduced'' is used).}
Let $j:\Cst_r(\GL)\to C_b(\GL)$, $j(a)(\gamma):=\langle\pi_{s(\gamma)}(a)\,\delta_{s(\gamma)},\delta_\gamma\rangle$. It is contractive and extends the identity on $C_\co$. Approximation
by $C_\co$ in the reduced norm shows that $j(a)$ is a uniform limit
of continuous compactly supported functions, hence belongs to
$C_0(\GL)$. It satisfies $\langle\pi_u(a)\delta_\alpha,\delta_\beta\rangle=j(a)(\beta\alpha^{-1})$ for all $\alpha,\beta\in\mathcal{G}_u$; hence $j$ is injective. A direct computation with $\xi=\delta_u$ gives, for every unit $u$,
\[
j(a^*a)(u)=\|\pi_u(a)\delta_u\|^2=\sum_{\gamma\in\mathcal{G}_u}|j(a)(\gamma)|^2 ,
\]
while $j(\widetilde P_g(a))=j(a)\cdot\mathbf 1_{c^{-1}(g)}$ (integrate $j(\alpha_z a)(\gamma)=z^{\,c(\gamma)}j(a)(\gamma)$ against $\bar z^{\,g}$). If $\widetilde P_0(a^*a)=0$ then the displayed sum vanishes for every $u$, so $j(a)\equiv0$ and $a=0$. Faithful. (Since each $c^{-1}(g)$ is compact, only boundedness of $j(a)$ is needed, not decay.)

\medskip\noindent\textbf{Case (ii).} Amplify the scalar algebra to
$M_n(C_r^*(\GL))$. The homogeneous pieces are
$C(\dom_g,M_n)$ with their supremum operator norms. Every amplified
coefficient map has the formula
\[
(\widetilde P_g\otimes\mathrm{id}_{M_n})(b)
=\int_{\T^N}\bar z^{\,g}(\alpha_z\otimes\mathrm{id}_{M_n})(b)\dd z,
\]
so its norm is at most one, uniformly in $n$. The zero coefficient
is a conditional expectation. If its value on $b^*b$ is zero, its
diagonal entries are sums of positive terms of the form
$\widetilde P_0(b_{ij}^*b_{ij})$; scalar faithfulness forces every
$b_{ij}$ to vanish. This proves matrix faithfulness as well.
\end{proof}

\begin{remark}[Fibres, symbols, and matrix coefficients]\label{rem:symbols}
The following notation from the proof of Theorem~\ref{thm:grading} is
used throughout.
\emph{(a) Fibres.} For $\omega\in\trans$ the source fibre is
parametrized by the configuration itself,
\[
\mathcal G_\omega:=s^{-1}(\omega)
=\big\{\gamma_x:=(\omega-x,\,-x):\ x\in\omega\big\}
\ \longleftrightarrow\ \omega,
\qquad
\gamma_x\gamma_y^{-1}=(\omega-x,\,y-x),
\]
so that $\pi_\omega$ acts on $\ell^2(\mathcal G_\omega)\cong\ell^2(\omega)$
and $c(\gamma_x\gamma_y^{-1})=a(y-x)$.
\emph{(b) Symbols.} The map $j:\mathfrak C\to C_b(\GL)$,
$j(a)(\gamma)=\langle\pi_{s(\gamma)}(a)\delta_{s(\gamma)},\delta_\gamma\rangle$,
of Step~5 is contractive and injective, restricts to the identity on
$C_\co(\GL)$, and satisfies
$j(\widetilde P_g a)=j(a)\,\mathbf 1_{c^{-1}(g)}$. For $a\in\mathfrak C$
set $a(\gamma):=j(a)(\gamma)$; then
$\langle\pi_\omega(a)\delta_{\gamma_y},\delta_{\gamma_x}\rangle
=a(\omega-x,\,y-x)$ for $x,y\in\omega$, i.e.\ $x,y\mapsto a(\omega-x,y-x)$
is the matrix of $\pi_\omega(a)$ in the basis $(\delta_x)_{x\in\omega}$.
\emph{(c) Matrix coefficients.} In case~(ii) every statement is read
with $\mathfrak C$ replaced by $M_n(\mathfrak C)\cong\Cst_r(\GL;M_n)$,
graded by the $M_n(\mathfrak C_g)$, with gauge projections
$\widetilde P_g\otimes\mathrm{id}_{M_n}$, still written
$\widetilde P_g$, and with $j$ then $M_n(\Cx)$-valued. When citing
\cite{JM} or \cite{SS}, their Fourier coefficient maps $E_g$ agree
with $\widetilde P_g$.
\end{remark}

\begin{lemma}[Faithfulness on a dense orbit]\label{lem:densefaithful}
Let $L$ be FLC and let $\omega\in\trans$ have a dense
$\GL$-orbit in $\trans$. Then the regular representation
$\pi_\omega$ of $C_r^*(\GL)$ is faithful.

In particular, if $L$ is repetitive, every regular representation
$\pi_\omega$ is faithful, without any aperiodicity assumption.
The same conclusions hold for matrix amplifications.
\end{lemma}

\begin{proof}
Write $E=\widetilde P_0$. Suppose that $\pi_\omega(b)=0$.
For every $\gamma\in\mathcal G_\omega=s^{-1}(\omega)$,
the matrix-coefficient formula gives
\[
E(b^*b)(r(\gamma))
=
\big\langle
\pi_\omega(b^*b)\delta_\gamma,\delta_\gamma
\big\rangle
=0.
\]
Thus the continuous function $E(b^*b)\in C(\trans)$ vanishes
on the orbit of $\omega$. By density it vanishes everywhere.
Faithfulness of $E$, proved in Theorem~\ref{thm:grading},
implies $b=0$.

If $L$ is repetitive, the groupoid is minimal by Step~5 of
Proposition~\ref{prop:basic}, so every orbit is dense.
Finally, matrix amplification preserves faithfulness.
\end{proof}

\begin{example}[Random binary alloys]\label{ex:random}
Let $\omega\in\{A,B\}^{\Z^d}$ have a Bernoulli distribution assigning
positive probability to both symbols. This is a decorated Delone set
(Remark~\ref{rem:decorated}), and its hull is almost surely the full
shift. FLC holds, so Theorem~\ref{thm:main}(i)--(ii) applies, but
minimality fails because the hull contains periodic configurations.
By Lemma~\ref{lem:densefaithful}, $\pi_\omega$ is faithful whenever
$\omega$ has a dense orbit, which holds almost surely. Invertibility
on such a configuration therefore transfers to the entire hull;
invertibility on an arbitrary configuration need not do so. This
example overlaps the setting of \cite{CLL}.
\end{example}

\begin{remark}[Structural consequence: the dual-action route]\label{rem:dualaction}
Step 4 exhibits more than a grading: a genuine continuous action of the compact abelian group $\T^N=\widehat{\Z^N}$ on the unital $\Cst$-algebra $\mathfrak{C}$ whose spectral subspaces are exactly the $\mathfrak{C}_g$. Consequently part (i) of the Main Theorem is a verbatim application of \cite[Thm.\ 2.8]{JM} (dual actions of compact abelian groups), which was written for precisely this mechanism. The weighted statement
is established separately in Proposition~\ref{prop:weightedSI}, while
the graded language also organizes the treatment of graded ideals
\cite[Thm.~2.16]{JM}.
\end{remark}

\begin{remark}[A trap to avoid]\label{rem:trap}
One should \emph{not} attempt to prove norm reconstruction $a=\sum_g\widetilde P_g(a)$ for all $a\in\mathfrak{C}$: it is false already for $\GL=\Z$ (then $\mathfrak{C}=C(\T)$ and the claim is uniform convergence of Fourier series of continuous functions, refuted by du Bois-Reymond). \cite[Def.\ 2.1]{JM} is calibrated exactly so as not to require it. For $a\in\ell^1(\mathfrak C)$, the reconstruction
$a=\sum_g\widetilde P_g(a)$ is absolutely convergent in
the C*-norm, hence unconditionally convergent.
Membership in $\ell^1(\mathfrak C)$ is characterized by
absolute summability of the coefficient norms, not here
by an assertion about unconditional convergence alone.
\end{remark}

\subsection{Spectral invariance in the Delone setting}
\begin{definition}[Weights]\label{def:weight}
A \emph{weight} on $\Z^N$ is a function $\nu:\Z^N\to[1,\infty)$ which
is submultiplicative, $\nu(g+h)\le\nu(g)\nu(h)$, symmetric,
$\nu(-g)=\nu(g)$, and normalized, $\nu(0)=1$ \cite[Def.~2.19]{JM}. It
satisfies the \emph{GRS condition} if $\lim_{n\to\infty}\nu(ng)^{1/n}=1$
for every $g\in\Z^N$; the limit exists by Fekete's subadditivity lemma.
Every weight in this paper is of this kind. Symmetry makes the
involution of $\ell^{1,\nu}$ isometric, and normalization gives
$\|\mathbf 1\|_{1,\nu}=1$.
\end{definition}

\begin{definition}\label{def:l1graded}
With the grading of Theorem~\ref{thm:grading} and a weight $\nu$ (Definition~\ref{def:weight}),
\[
\ell^{1,\nu}(\mathfrak{C}):=\Big\{\Phi\in\mathfrak{C}\ :\ \|\Phi\|_{1,\nu}:=\sum_{g\in\Z^N}\nu(g)\,\big\|\widetilde P_g(\Phi)\big\|_{\mathfrak{C}}<\infty\Big\},
\qquad \ell^{1}(\mathfrak{C}):=\ell^{1,1}(\mathfrak{C}).
\]
\end{definition}

\begin{theorem}[Main Theorem]\label{thm:main}
Let $L\subset\R^d$ be an FLC Delone set and $\Fell$ admissible. Then:
\begin{enumerate}[label=\emph{(\roman*)},itemsep=2pt]
\item $\ell^{1}(\mathfrak{C})$ is a reduced, symmetric Banach $*$-algebra, inverse closed in $\mathfrak{C}$;
\item if $\nu$ is a GRS weight, then $\ell^{1,\nu}(\mathfrak{C})$ is a symmetric Banach $*$-algebra, inverse closed in $\mathfrak{C}$, and spectrally invariant in every faithful nondegenerate representation
(Proposition~\ref{prop:weightedSI});
\item if moreover $L$ is repetitive, then for \emph{every} $\omega\in\trans$ the regular representation $\pi_\omega$ on $\ell^2(\omega)\otimes\Cx^n$ is faithful and unital, and consequently, for every GRS weight $\nu$, the algebra $\pi_\omega\big[\ell^{1,\nu}(\mathfrak{C})\big]$ is inverse closed in $\Bd\big(\ell^2(\omega)\otimes\Cx^n\big)$.
\end{enumerate}
\end{theorem}

\noindent The proof is given at the end of this subsection, after the supporting Lemmas~\ref{lem:radtransfer} and~\ref{lem:hulanicki} and Proposition~\ref{prop:weightedSI}.

The weighted theorem \cite[Thm.~2.20]{JM} is formulated using
uniform GRS and shell-comparability assumptions
\cite[(2.11)--(2.12)]{JM}. For the present grading group,
ordinary GRS already implies uniform GRS. Indeed, for every
$\varepsilon>0$, submultiplicativity and GRS on the standard basis
give a constant $C_\varepsilon$ such that
\[
  \nu(g)\le C_\varepsilon e^{\varepsilon\|g\|_1}
  \qquad(g\in\Z^N).
\]
Every word ball for a finite generating set is contained in an
$\ell^1$-ball of radius proportional to its word radius, which proves
the uniform condition. Shell comparability is not assumed.
The following elementary transfer argument combines scalar
Beurling-algebra spectral theory with unweighted graded spectral
invariance. It handles the exact, possibly non-radial physical weights
without a shell assumption; no priority claim is made for the GRS
principle.

Throughout, $\mathfrak B$ denotes a unital C*-algebra, topologically
graded over $\Z^N$ with closed components $(\mathfrak B_g)_{g\in\Z^N}$
and $1\in\mathfrak B_0$. The summable-coordinate definition
of $\ell^{1,\nu}(\mathfrak B)$, with the norms of the homogeneous
components inherited from $\mathfrak B$; compare \cite[\S2]{JM}.
The following elementary observation fixes the Banach-algebra and
inclusion properties used in the spectral-radius arguments.

\begin{lemma}[Weighted coordinate algebra]\label{lem:weightedbanach}
For a normalized symmetric submultiplicative weight $\nu$, the space
\[
\mathcal A_\nu=
\left\{(b_g)_{g\in\Z^N}:b_g\in\mathfrak B_g,\quad
\sum_g\nu(g)\|b_g\|_{\mathfrak B}<\infty\right\}
\]
is a unital Banach $*$-algebra for
\[
(bd)_g=\sum_h b_h d_{g-h},\qquad (b^*)_g=b_{-g}^*,
\]
with isometric involution and unit of norm one. The summation map
$\iota(b)=\sum_g b_g$ is a continuous injective unital
$*$-homomorphism into $\mathfrak B$, and
$\|\iota(b)\|_{\mathfrak B}\le\|b\|_{\mathcal A_\nu}$.
Its image is dense in $\mathfrak B$.
For the concrete grading of Theorem~\ref{thm:grading}, this image
is exactly the coefficient-defined space in
Definition~\ref{def:l1graded}, with the stated norm.
\end{lemma}

\begin{proof}
The weighted $\ell^1$-sum of the closed Banach spaces
$\mathfrak B_g$ is complete. Each product-coordinate series
converges absolutely in the closed space $\mathfrak B_g$, and
Tonelli's theorem and submultiplicativity give
\[
\sum_g\nu(g)\|(bd)_g\|
\le\sum_{g,h}\nu(h)\nu(g-h)\|b_h\|\|d_{g-h}\|
=\|b\|_{\mathcal A_\nu}\|d\|_{\mathcal A_\nu}.
\]
The involution and unit assertions follow from symmetry and
normalization of $\nu$. Since $\nu\ge1$, summation converges
absolutely in $\mathfrak B$ and respects both operations.
The continuous coefficient projections satisfy
$\widetilde P_h(\iota(b))=b_h$, proving injectivity. Finite
homogeneous sums are contained in the image and are dense in
$\mathfrak B$.
For the concrete coefficient definition, if
$\sum_g\nu(g)\|\widetilde P_g(a)\|<\infty$, set
$b=\sum_g\widetilde P_g(a)$ in $\mathfrak C$. Every coefficient
of $a-b$ vanishes. The symbol identity from Step~5 of
Theorem~\ref{thm:grading} gives
$j(a-b)1_{c^{-1}(g)}=0$ for all $g$; injectivity of $j$ yields
$a=b$. This proves the final identification, without assuming
Fourier-series convergence for arbitrary elements of $\mathfrak C$.
\end{proof}

The image of $\mathcal A_\nu$ is denoted by
$\ell^{1,\nu}(\mathfrak B)$. No GRS assumption was needed in this
lemma. The case of interest is the algebra $\mathfrak C$ from
Convention~\ref{conv:coeffscope}.

\begin{lemma}[Spectral radius transfer]\label{lem:radtransfer}
Let $\mathfrak B$ be a unital C*-algebra topologically graded over
$\Z^N$, with the weighted coordinate algebras of
Lemma~\ref{lem:weightedbanach}, and let $\nu$ be a GRS weight.
Then, for every $\Phi\in\ell^{1,\nu}(\mathfrak B)$,
\[
r_{\ell^{1,\nu}(\mathfrak B)}(\Phi)=r_{\ell^{1}(\mathfrak B)}(\Phi).
\]
\end{lemma}

\begin{proof}
Since $\nu\ge1$, $r_{\ell^1}\le r_{\ell^{1,\nu}}$. For the converse put
$\phi(g):=\|\Phi_g\|$. The grading and submultiplicativity give
$\|(\Phi\Psi)_g\|\le(\phi*\psi)(g)$, hence
$\|\Phi^n\|_{\ell^{1,\nu}}\le\|\phi^{*n}\|_{\ell^1_\nu(\Z^N)}$ and
$r_{\ell^{1,\nu}(\mathfrak B)}(\Phi)\le r_{\ell^1_\nu(\Z^N)}(\phi)$.
For clarity, the scalar Beurling-algebra step can be checked directly;
compare \cite{FGL}. If $\chi$ is a character of $\ell^1_\nu(\Z^N)$,
then
\[
|\chi(\delta_g)|^m=|\chi(\delta_{mg})|\le\nu(mg).
\]
GRS gives $|\chi(\delta_g)|\le1$; applying the same argument to
$-g$ and using $\delta_g*\delta_{-g}=\delta_0$ gives equality.
Thus $\chi(\delta_g)=z^g$ for some $z\in\T^N$. Conversely,
every such evaluation is continuous because $\nu\ge1$.
Density of finite supports identifies all characters, so for
$\phi\ge0$,
\[
r_{\ell^1_\nu(\Z^N)}(\phi)=\max_{z\in\T^N}\Big|\sum_g\phi(g)z^g\Big|
=\sum_g\phi(g)=\|\Phi\|_{\ell^1(\mathfrak B)}.
\]
Thus $r_{\ell^{1,\nu}}(\Psi)\le\|\Psi\|_{\ell^1}$ for every
$\Psi\in\ell^{1,\nu}(\mathfrak B)$; applied to $\Psi=\Phi^m$,
\[
r_{\ell^{1,\nu}}(\Phi)=r_{\ell^{1,\nu}}(\Phi^m)^{1/m}
\le\|\Phi^m\|_{\ell^1}^{1/m}\xrightarrow[m\to\infty]{}r_{\ell^1}(\Phi).
\qedhere
\]
\end{proof}

The scalar case is classical Beurling-algebra theory
\cite{FGL}. Weighted operator-valued Wiener theorems also precede
this formulation; see \cite{Baskakov97,KB25}. The role of
Lemma~\ref{lem:radtransfer} here is to provide a direct graded proof.
The next criterion is the classical Hulanicki lemma \cite{Hul};
its proof is included for completeness.
\begin{lemma}[{Hulanicki \cite{Hul}}]\label{lem:hulanicki}
Let $\mathcal{A}$ be a unital Banach $*$-algebra, $\mathcal{B}$ a
unital C*-algebra, and $\iota:\mathcal{A}\to\mathcal{B}$ an injective
unital $*$-homomorphism such that
\[
r_{\mathcal{A}}(a)=\|\iota(a)\|_{\mathcal{B}}
\qquad\text{for every }a=a^*\in\mathcal{A}.
\]
Then $\mathrm{sp}_{\mathcal{A}}(a)=\mathrm{sp}_{\mathcal{B}}(\iota(a))$
for \emph{every} $a\in\mathcal{A}$; in particular $\iota(\mathcal{A})$
is inverse-closed in $\mathcal{B}$.
\end{lemma}

\begin{proof}
The inclusion
$\mathrm{sp}_{\mathcal{B}}(\iota a)\subseteq\mathrm{sp}_{\mathcal{A}}(a)$
holds for any unital homomorphism, so only the converse needs proof.

\emph{Self-adjoint case.} Let $a=a^*$ and
$\lambda\notin\mathrm{sp}_{\mathcal{B}}(\iota a)$. Since
$\mathrm{sp}_{\mathcal{B}}(\iota a)\subseteq\R$, also
$\bar\lambda\notin\mathrm{sp}_{\mathcal{B}}(\iota a)$. Put
$b:=(a-\bar\lambda)(a-\lambda)\in\mathcal{A}$; then $b=b^*$ and
$\iota b=(\iota a-\lambda)^*(\iota a-\lambda)$ is positive and
invertible in $\mathcal{B}$, so
$\mathrm{sp}_{\mathcal{B}}(\iota b)\subseteq[\delta,\|\iota b\|]$ for
some $\delta>0$. The element $c:=\|\iota b\|\,1-b$ is self-adjoint with
$\mathrm{sp}_{\mathcal{B}}(\iota c)\subseteq[0,\|\iota b\|-\delta]$, so
the hypothesis yields
$r_{\mathcal{A}}(c)=\|\iota c\|_{\mathcal{B}}\le\|\iota b\|-\delta$.
Every $w\in\mathrm{sp}_{\mathcal{A}}(b)
=\|\iota b\|-\mathrm{sp}_{\mathcal{A}}(c)$ therefore satisfies
$\big|\,\|\iota b\|-w\,\big|\le\|\iota b\|-\delta<\|\iota b\|$, whence
$0\notin\mathrm{sp}_{\mathcal{A}}(b)$ and $b$ is invertible in
$\mathcal{A}$. As $b$ is a polynomial in $a$, its inverse commutes with
$a$, and $b^{-1}(a-\bar\lambda)$ is then a two-sided inverse of
$a-\lambda$ in $\mathcal{A}$. Thus
$\lambda\notin\mathrm{sp}_{\mathcal{A}}(a)$.

\emph{General case.} If $\iota x$ is invertible in $\mathcal{B}$, so
are $\iota(x^*x)$ and $\iota(xx^*)$; by the self-adjoint case, $x^*x$
and $xx^*$ are invertible in $\mathcal{A}$, and $(x^*x)^{-1}x^*$,
$x^*(xx^*)^{-1}$ are a left and a right inverse of $x$, so $x$ is
invertible in $\mathcal{A}$. Applying this to $x-\lambda$ gives
$\mathrm{sp}_{\mathcal{A}}(x)\subseteq\mathrm{sp}_{\mathcal{B}}(\iota x)$.
\end{proof}

\begin{proposition}[Weighted spectral invariance over $\Z^N$]\label{prop:weightedSI}
Let $\mathfrak{B}$ be a unital C*-algebra topologically graded over
$\Z^N$ and let $\nu$ be a GRS weight. Then:
\begin{enumerate}
\item[(i)] $\ell^{1,\nu}(\mathfrak{B})$ is inverse-closed in
$\mathfrak{B}$; more precisely,
$\mathrm{sp}_{\ell^{1,\nu}(\mathfrak{B})}(\Phi)
=\mathrm{sp}_{\mathfrak{B}}(\Phi)$ for every
$\Phi\in\ell^{1,\nu}(\mathfrak{B})$.
\item[(ii)] For every faithful nondegenerate representation $\Pi$ of
$\mathfrak{B}$ on a Hilbert space $\mathcal{H}$, the algebra
$\Pi\big[\ell^{1,\nu}(\mathfrak{B})\big]$ is inverse-closed in
$B(\mathcal{H})$.
\item[(iii)] $\ell^{1,\nu}(\mathfrak{B})$ is a symmetric Banach
$*$-algebra, stable under the holomorphic functional calculus of
$\mathfrak{B}$.
\end{enumerate}
In particular all of this holds for
$\mathfrak{B}=\mathfrak{C}=C^*_r(\GL)$ with the grading of Theorem~\ref{thm:grading}
and any GRS weight. By Lemma~\ref{lem:densefaithful}, (ii) applies
whenever the chosen configuration has a dense orbit; for repetitive
$L$ it applies to every configuration, with the same decay class.
\end{proposition}

\begin{proof}
(i) For $\Phi=\Phi^*\in\ell^{1,\nu}(\mathfrak{B})$,
Lemma~\ref{lem:radtransfer} and the inverse-closedness of the
\emph{unweighted} algebra $\ell^{1}(\mathfrak{B})$ in $\mathfrak{B}$
\cite[Thm.~2.4(ii)]{JM} give
\[
r_{\ell^{1,\nu}}(\Phi)
=r_{\ell^{1}}(\Phi)
=r_{\mathfrak{B}}(\Phi)
=\|\Phi\|_{\mathfrak{B}},
\]
the last equality being the C*-identity for self-adjoint elements.
Lemma~\ref{lem:hulanicki}, applied to
$\iota:\ell^{1,\nu}(\mathfrak{B})\to\mathfrak{B}$, yields (i).

(ii) A faithful representation of a C*-algebra is isometric, and
$\Pi(\mathfrak{B})\subseteq B(\mathcal{H})$ is a C*-subalgebra
containing $I_{\mathcal{H}}=\Pi(1)$; by spectral permanence of
C*-subalgebras,
$\mathrm{sp}_{B(\mathcal{H})}(\Pi\Phi)=\mathrm{sp}_{\mathfrak{B}}(\Phi)$,
and (i) transfers inverses back into
$\Pi\big[\ell^{1,\nu}(\mathfrak{B})\big]$.

(iii) By (i), self-adjoint elements of $\ell^{1,\nu}(\mathfrak{B})$
have real spectra, so the algebra is hermitian and, by the
Shirali--Ford theorem \cite{SF}, symmetric. For the functional calculus,
if $f$ is holomorphic on a neighbourhood of
$\mathrm{sp}_{\mathfrak{B}}(\Phi)=\mathrm{sp}_{\ell^{1,\nu}}(\Phi)$,
the Riesz--Dunford integral
$\tfrac{1}{2\pi i}\oint f(\lambda)\,(\lambda-\Phi)^{-1}\,d\lambda$
converges in $\ell^{1,\nu}(\mathfrak{B})$. By (i), the resolvent is
$\ell^{1,\nu}$-valued and norm-continuous in $\lambda$. The map $\iota$
then sends the integral to $f(\iota\Phi)$ computed in $\mathfrak{B}$.
\end{proof}

\begin{corollary}[Matrix stability and K-theory]\label{cor:matrixK}
For every $n\in\mathbb{N}$, the C*-algebra $M_n(\mathfrak{B})$ is
topologically $\Z^N$-graded with components $M_n(\mathfrak{B}_g)$, and
$M_n\big(\ell^{1,\nu}(\mathfrak{B})\big)
=\ell^{1,\nu}\big(M_n(\mathfrak{B})\big)$ is inverse-closed in
$M_n(\mathfrak{B})$. Consequently $\ell^{1,\nu}(\mathfrak{B})$ is a
dense $*$-subalgebra of $\mathfrak{B}$, spectrally invariant and stable
under holomorphic functional calculus together with all its matrix
amplifications, and the inclusion induces isomorphisms
\[
K_*\big(\ell^{1,\nu}(\mathfrak{B})\big)
\;\xrightarrow{\ \cong\ }\;
K_*(\mathfrak{B})
\]
by the density theorem \cite[Thm.~1.3.1]{Bost}; see also the explicit formulation in \cite[Thm.~1.1]{Nica08}. In particular
$K_*\big(\ell^{1,\nu}(\mathfrak{C})\big)\cong K_*\big(C^*_r(\GL)\big)$:
the $K$-groups can be represented in the weighted algebra after
matrix stabilization: $K_0$-classes by differences of projection
classes and $K_1$-classes by invertible representatives. For a Fermi
projection obtained from a Hamiltonian in the weighted algebra,
membership follows directly from holomorphic functional calculus
across the gap. No new gap-label set is computed here.
\end{corollary}

\begin{proof}
The map $\widetilde P_0\otimes\mathrm{id}_{M_n}$ is a conditional expectation onto
$M_n(\mathfrak{B}_0)$ vanishing on $M_n(\mathfrak{B}_g)$ for $g\ne0$;
the algebraic sum of the $M_n(\mathfrak{B}_g)$ is dense in
$M_n(\mathfrak{B})$, and matrix products and adjoints respect the
grading entrywise. Hence $M_n(\mathfrak{B})$ is topologically
$\Z^N$-graded, and Proposition~\ref{prop:weightedSI} applies verbatim
to it: the input \cite[Thm.~2.4(ii)]{JM} holds for $M_n(\mathfrak{B})$
because $\Z^N$ is rigidly symmetric, exactly as in the proof of
Theorem~\ref{thm:main}(i). The identification
$M_n\big(\ell^{1,\nu}(\mathfrak{B})\big)
=\ell^{1,\nu}\big(M_n(\mathfrak{B})\big)$ holds with equivalence of
norms, by entrywise estimates with constants depending only on $n$.
The K-theory statement is the density theorem applied to the
continuous, dense, spectrally invariant (together with matrices)
inclusion $\iota$. Projection representatives may be obtained by
approximating an ambient projection by a self-adjoint element of the
weighted matrix algebra and applying the Riesz calculus on a small
contour around $1$. Sufficiently close projections determine the
same $K_0$-class. This explains the projection formulation, rather
than only a formulation in terms of Banach-algebra idempotents.
\end{proof}

\begin{remark}[Scope, envelope, and what is not claimed]\label{rem:noshell}
(i) \emph{Over $\Z^N$, GRS suffices.} Proposition~\ref{prop:weightedSI}
requires only the GRS condition. When the grading group is $\Z^N$, the
shell-comparability hypothesis \cite[(2.12)]{JM} in
\cite[Thm.~2.20]{JM} is not needed for symmetry or spectral invariance.
The uniform GRS condition in \cite[(2.11)]{JM} also holds automatically.
For general rigidly symmetric groups, both assumptions remain natural
because the commutative argument in the proof of
Lemma~\ref{lem:radtransfer} is unavailable. The weights considered here,
$(1+|g|)^s$ and $e^{\lambda|g|^{\alpha}}$ for $0<\alpha<1$, satisfy
GRS. Fixed exponential weights $e^{\lambda|g|}$ do not. The shift
counterexample in Remark~\ref{rem:weights} rules out a general
inverse-closedness statement for those weights, but it does not imply
failure for every individual graded algebra.
(ii) \emph{The envelope, a posteriori.} Since
$\ell^{1,\nu}(\mathfrak{B})$ is symmetric, its greatest C*-seminorm is
the Pt\'ak function $\rho(\Phi)=r(\Phi^*\Phi)^{1/2}$ \cite{Ptak}
, which by Lemma~\ref{lem:radtransfer} and
\cite[Thm.~2.4(ii)]{JM} equals $\|\iota\Phi\|_{\mathfrak{B}}$; hence
$C^*_{\mathrm{env}}\big(\ell^{1,\nu}(\mathfrak{B})\big)=\mathfrak{B}$,
which closes, a posteriori, the conceptual gap described above.
(iii) \emph{Norm control.} The proposition asserts membership of
inverses with no a priori bound on $\|\Phi^{-1}\|_{\ell^{1,\nu}}$ in
terms of $\|\Phi\|_{\ell^{1,\nu}}$ and
$\|\Phi^{-1}\|_{\mathfrak B}$. In the Delone setting, quantitative
bounds are available for polynomial weights by Theorem~\ref{thm:nci}
and for stretched-exponential weights by
Theorem~\ref{thm:stretchednci}. No norm-control assertion is made
here for arbitrary GRS weights.
(iv) \emph{Nonunital case.} Everything above is stated for unital
$\mathfrak{B}$ (in the present case, the transversal is compact, so $\mathfrak{C}$
is unital). If ever needed, the nonunital case follows by unitization:
$\widetilde{\mathfrak{B}}$ inherits the grading with
$(\widetilde{\mathfrak{B}})_0=\mathfrak{B}_0+\mathbb{C}1$, and the
arguments pass through unchanged.
\end{remark}

\begin{proof}[Proof of Theorem~\ref{thm:main}]
(i) Step 4 of the proof of Theorem~\ref{thm:grading} implements the
grading by a continuous gauge action of the compact abelian group
$\T^N=\widehat{\Z^N}$ on the unital $\Cst$-algebra $\mathfrak{C}$,
with spectral subspaces $\mathfrak{C}_g$. Part (i) therefore follows
directly from \cite[Thm.\ 2.8]{JM}; see also
Remark~\ref{rem:dualaction}. Alternatively, it follows from
\cite[Thm.\ 2.4]{JM}, since $\Z^N$ is abelian and hence rigidly
symmetric by Leptin--Poguntke. Part (ii) follows from
Proposition~\ref{prop:weightedSI}, applied to $\mathfrak{B}=\mathfrak{C}$
with the grading of Theorem~\ref{thm:grading}.

(iii) If $L$ is repetitive, minimality and
Lemma~\ref{lem:densefaithful} imply that every $\pi_\omega$
is faithful. Proposition~\ref{prop:weightedSI}(ii) then transfers
inverse-closedness to each regular representation.
\end{proof}

\begin{remark}[Weights: what is and is not reachable]\label{rem:weights}
Polynomial and stretched-exponential weights satisfy GRS, whereas
$e^{c|g|}$ with $c>0$ does not. Fixed exponential weights are not
inverse-closed in general: already for the bilateral shift $U$,
$\mathbf 1-tU$ with $e^{-c}<t<1$ has finite support but its inverse
$\sum_{m\ge0}t^mU^m$ is not summable with weight $e^{c|m|}$.

For general GRS weights the conclusion is preservation of the given
weighted class under inversion, products and holomorphic functional
calculus, with a hull-uniform profile. Positive-temperature
Fermi--Dirac functional calculus needs no spectral gap. A
zero-temperature Fermi projection requires a gap in the C*-spectrum
and retains the available class; a gap alone does not improve decay
to exponential. Exponential localization requires additional locality
hypotheses and corresponding resolvent estimates \cite{BBR,CO,OTC}.
Norm control is proved here for polynomial and stretched-exponential
weights, not for all GRS weights.
\end{remark}

\subsection{Graded ideals and defects: conditional outlook}

\begin{remark}[Defect ideals: conditional formulation]
\label{prop:fredholm}
Let $L'$ be FLC and suppose that
$F\subseteq\trans(L')$ is a nonempty closed invariant subset.
Put
\[
\mathcal G'=\mathcal G_{L'},\qquad
U=\trans(L')\setminus F,\qquad
\mathfrak J=C_r^*(\mathcal G'|_U).
\]
Amenability gives the exact sequence
\cite[Prop.~4.3.2 and its reduced-algebra discussion]{Sims18}
\[
0\longrightarrow\mathfrak J
\longrightarrow C_r^*(\mathcal G')
\longrightarrow C_r^*(\mathcal G'|_F)
\longrightarrow0.
\]
The ideal $\mathfrak J$ is graded: compactly supported sections
on $\mathcal G'|_U$ split into finitely many homogeneous
components, still supported on that reduction.
By \cite[Thm.~2.16]{JM}, the image of
$\ell^1(C_r^*(\mathcal G'))$ in the quotient is inverse-closed.

For a defect model, identifying $F$ with the transversal of
the background hull is an additional geometric step.
The closed invariant subset defines the quotient; its open
complement defines the ideal.

A Fredholm interpretation further requires a faithful
representation $\Pi$ with
$\Pi(\mathfrak J)=\mathcal K(\mathcal H)$; compare
\cite[Cor.~2.18]{JM}.
No such identification is assumed automatically.
\end{remark}

\section{The kernel dictionary}\label{sec:kernels}

This section translates the algebraic results into the language of covariant kernels used in the physics and numerical-analysis literature. It mirrors \cite[\S3.6]{FJM} and \cite[Def.\ 2.5, Cor.\ 2.6]{JM}.

\begin{definition}[Covariant convolution-dominated kernels]\label{def:cd}
A \emph{hull-covariant kernel} is a family $K=(K_\omega)_{\omega\in\trans}$
of functions $K_\omega:\omega\times\omega\to M_n(\Cx)$ such that:
\begin{enumerate}[label=\emph{(\alph*)},itemsep=2pt]
\item \emph{(covariance)}
$K_{\omega-t}(x-t,\,y-t)=K_\omega(x,y)$ for all $\omega\in\trans$,
$t\in\omega$ and $x,y\in\omega$;
\item \emph{(local dependence)} for every $x,y\in\R^d$ the function
$\omega\mapsto K_\omega(x,y)$ is continuous on the set
$\dom_x\cap\dom_y\subseteq\trans$, clopen by \eqref{eq:domclopen}.
\end{enumerate}
Given a weight $\nu$ on $\Z^N$ (Definition~\ref{def:weight}),
$K\in\mathcal{K}^\nu_{\mathrm{cov}}(L)$ means that, in addition, there is
a nonnegative function $k:\Z^N\to[0,\infty)$ in
$\ell^{1,\nu}(\Z^N)$ with
\[
\|K_\omega(x,y)\|\ \le\ k\big(a(x-y)\big)
\qquad\text{for all }\omega\in\trans,\ x,y\in\omega
\]
The inequality is meaningful because $x-y\in\Gamma_L$ by
Lemma~\ref{lem:hullcompat}. We define
$\|K\|_{\mathcal{K}^\nu}:=\inf\big\{\|k\|_{\ell^{1,\nu}}:\ k\text{ as
above}\big\}$.
\end{definition}

The definition is intrinsic: no groupoid enters. The bridge to $\GL$ is
the following lemma, which is where finite local complexity acts.

\begin{lemma}[Covariant kernels are functions on $\GL$]\label{lem:transport}
For a family $K=(K_\omega)_{\omega\in\trans}$ satisfying the covariance
\emph{(a)}, define
\[
\Phi_K:\GL\longrightarrow M_n(\Cx),\qquad
\Phi_K(\omega,x):=K_\omega(0,x).
\]
Then:
\begin{enumerate}[label=\emph{(\roman*)},itemsep=2pt]
\item the assignments $K\mapsto\Phi_K$ and $\Phi\mapsto K^\Phi$,
$K^\Phi_\omega(x,y):=\Phi(\omega-x,\,y-x)$, are mutually inverse
bijections between the families satisfying \emph{(a)} and the functions
$\GL\to M_n(\Cx)$;
\item under this bijection, $K$ satisfies the local dependence
\emph{(b)} if and only if $\Phi_K$ is continuous on $\GL$.
\end{enumerate}
\end{lemma}

\begin{proof}
(i) All expressions are defined: $(\omega,x)\in\GL$ gives $0,x\in\omega$,
while $x\in\omega$ gives $\omega-x\in\trans$ and $0,\,y-x\in\omega-x$.
Any $K^\Phi$ satisfies (a):
\[
K^\Phi_{\omega-t}(x-t,\,y-t)
=\Phi\big((\omega-t)-(x-t),\,(y-t)-(x-t)\big)
=\Phi(\omega-x,\,y-x)=K^\Phi_\omega(x,y).
\]
Moreover $\Phi_{K^\Phi}(\omega,x)=K^\Phi_\omega(0,x)=\Phi(\omega,x)$;
and if $K$ satisfies (a), then, taking $t=x$ there,
\[
K^{\Phi_K}_\omega(x,y)=\Phi_K(\omega-x,\,y-x)
=K_{\omega-x}(0,\,y-x)=K_{\omega-x}(x-x,\,y-x)=K_\omega(x,y).
\]

(ii) Suppose (b). The sets $c^{-1}(g)$, $g\in a(L-L)$, form a partition
of $\GL$ into clopen sets, on each of which the second coordinate is
the constant $x_g$ and $r$ restricts to a homeomorphism onto $\dom_g$
(Proposition~\ref{prop:cocycle}). On $c^{-1}(g)$,
\[
\Phi_K=\big[\omega\mapsto K_\omega(0,x_g)\big]\circ r|_{c^{-1}(g)},
\]
continuous by (b) applied to the pair $(0,x_g)$, whose domain is
$\dom_0\cap\dom_{x_g}=\trans\cap\dom_g=\dom_g$. Being continuous on
every member of a clopen partition, $\Phi_K$ is continuous on $\GL$.

Conversely, suppose $\Phi_K$ continuous and fix $x,y\in\R^d$ with
$\dom_x\cap\dom_y\neq\varnothing$ (for other pairs (b) is vacuous).
Then $x,y\in\omega$ for some $\omega\in\trans$, so
$y-x\in\omega-\omega\subseteq\Gamma_L$ (Lemma~\ref{lem:hullcompat})
and $g:=a(y-x)$ is defined, with $x_g=y-x$. By (i), on
$\dom_x\cap\dom_y$,
\[
K_\omega(x,y)=\Phi_K(\omega-x,\,y-x)
=\big(\Phi_K\circ\iota_g\circ T_x\big)(\omega),
\]
where $T_x:\omega\mapsto\omega-x$ is the translation homeomorphism
$\dom_x\to\dom_{-x}$ (Step~2 of the proof of
Proposition~\ref{prop:basic}) and
$\iota_g:\eta\mapsto(\eta,\,x_g)$ is the inverse of
$r|_{c^{-1}(g)}$, continuous by Proposition~\ref{prop:cocycle}(iii).
The composition is legitimate: if $\omega\in\dom_x\cap\dom_y$ then
$y-x\in\omega-x$, i.e.\ $T_x(\omega)\in\dom_{x_g}=\dom_g$. Hence
$\omega\mapsto K_\omega(x,y)$ is continuous on $\dom_x\cap\dom_y$,
which is (b).
\end{proof}

\begin{remark}[Local dependence; pattern-equivariant kernels]\label{rem:localdep}
Condition \emph{(b)} says that each matrix element $K_\omega(x,y)$
depends continuously on the configuration $\omega$ in the local
matching topology: it is determined to any prescribed accuracy by a
sufficiently large patch of $\omega$.
If, for each fixed pair $(x,y)$, the matrix element depends
exactly on a sufficiently large finite patch, then the
corresponding coefficient function is locally constant, as in
strongly pattern-equivariant constructions \cite{KellendonkPutnam00}.
This finite-patch dependence must be distinguished from
finite spatial range, which means that there exists $R_0$
such that
\[
K_\omega(x,y)=0
\qquad\text{whenever }|x-y|>R_0,
\]
uniformly over the hull.
Neither condition is part of the definition of continuous
local dependence.
Tight-binding Hamiltonians and overlap matrices whose entries are
computed from the local atomic environment satisfy \emph{(b)} by
construction. By covariance, imposing \emph{(b)} only for the pairs
$(0,x)$, $x\in\R^d$, already yields the general case, as used in the
proof of Lemma~\ref{lem:transport}(ii); the symmetric formulation is
retained.
\end{remark}

\begin{proposition}[Bisection lemma and the dictionary]\label{prop:dictionary}
Let $L$ be FLC and $\nu$ a weight (Definition~\ref{def:weight});
we use the notation of Remark~\ref{rem:symbols}, where
$\Phi(\gamma)$ abbreviates $j(\Phi)(\gamma)$. In case \emph{(ii)},
all statements are understood at the level of $M_n(\mathfrak C)$.
\begin{enumerate}[label=\emph{(\roman*)},itemsep=2pt]
\item For every $\Phi\in\mathfrak{C}$ and $g\in\Z^N$,
\[
\big\|\widetilde P_g(\Phi)\big\|_{\mathfrak{C}}
\ =\ \sup_{\omega\in\dom_g}\big\|\Phi(\omega,x_g)\big\|,
\]
the supremum being attained on the compact set $\dom_g$
(with $\sup\varnothing:=0$).
\item The map $\Phi\mapsto K^{\Phi}$,
\[
K^{\Phi}_\omega(x,y)\ :=\ \Phi(\omega-x,\,y-x)
\qquad(\omega\in\trans,\ x,y\in\omega),
\]
is an isometric bijection of $\ell^{1,\nu}(\mathfrak{C})$ onto
$\mathcal{K}^{\nu}_{\mathrm{cov}}(L)$ intertwining the product and
involution of $\ell^{1,\nu}(\mathfrak{C})$ with the kernel operations
\begin{align*}
(K\!\cdot\!K')_\omega(x,y)
&=\sum_{z\in\omega}K_\omega(x,z)\,K'_\omega(z,y)
\quad\text{(absolutely convergent)},\\
K^{*}_\omega(x,y)&:=K_\omega(y,x)^{*}.
\end{align*}
Consequently $\mathcal{K}^{\nu}_{\mathrm{cov}}(L)$, with these
operations, is a unital Banach $*$-algebra whose unit is the Kronecker
kernel $\delta_{x,y}\mathbf 1$. The map $\Phi\mapsto K^\Phi$ is an
isometric $*$-isomorphism. The infimum defining
$\|K^{\Phi}\|_{\mathcal K^\nu}$ is attained, at
$k_0(g)=\|\widetilde P_{-g}(\Phi)\|_{\mathfrak{C}}$.
\item For every $\omega\in\trans$, $K^{\Phi}_\omega$ is the integral
kernel of $\pi_\omega(\Phi)$ on $\ell^2(\omega)\otimes\Cx^n$; in
particular every $K\in\mathcal{K}^{\nu}_{\mathrm{cov}}(L)$ acts on each
$\ell^2(\omega)\otimes\Cx^n$ as a bounded operator, of norm at most
$\|K\|_{\mathcal K^\nu}$.
\end{enumerate}
\end{proposition}

\begin{proof}
\textbf{Step 1: (i).} $\widetilde P_g(\Phi)$ lies in
$\mathfrak C_g=C_\co(c^{-1}(g))$, on which the reduced norm is the sup
norm (Step~2 of the proof of Theorem~\ref{thm:grading}; in case (ii),
$M_n(\mathfrak C_g)\cong C(\dom_g,M_n(\Cx))$ with the supremum of
matrix norms). Since $j$ restricts to the identity on $C_\co(\GL)$ and
$j(\widetilde P_g\Phi)=j(\Phi)\,\mathbf 1_{c^{-1}(g)}$ (Step~5 there),
for $\omega\in\dom_g$ one has
$(\widetilde P_g\Phi)(\omega,x_g)=\Phi(\omega,x_g)$; as $r$ identifies
$c^{-1}(g)$ with $\dom_g$ (Proposition~\ref{prop:cocycle}(iii)),
\[
\|\widetilde P_g(\Phi)\|_{\mathfrak C}
=\sup_{\gamma\in c^{-1}(g)}\big\|(\widetilde P_g\Phi)(\gamma)\big\|
=\sup_{\omega\in\dom_g}\big\|\Phi(\omega,x_g)\big\|,
\]
attained by compactness of $\dom_g$ and continuity.

\textbf{Step 2: kernels of $\pi_\omega$, operations, injectivity.}
By Remark~\ref{rem:symbols}(a)--(b), for $\Phi\in\mathfrak C$ and
$x,y\in\omega$,
\[
\big\langle \pi_\omega(\Phi)\,\delta_{\gamma_y},\delta_{\gamma_x}\big\rangle
= j(\Phi)\big(\gamma_x\gamma_y^{-1}\big)
= \Phi(\omega-x,\,y-x)
= K^{\Phi}_\omega(x,y):
\]
$K^\Phi_\omega$ is the matrix of $\pi_\omega(\Phi)$ in the basis
$(\delta_x)_{x\in\omega}$. This is (iii), with
$\|\pi_\omega(\Phi)\|\le\|\Phi\|_{\mathfrak C}\le\|\Phi\|_{1,\nu}$ for
$\Phi\in\ell^{1,\nu}(\mathfrak C)$, since $\nu\ge1$. The matrix of a
product of bounded operators is the matrix product of their matrices,
the sum $\sum_{z\in\omega}$ converging unconditionally because
$\sum_{z}P_z=\mathbf 1$ strongly ($P_z$ the coordinate projections);
as $\pi_\omega$ is a $*$-homomorphism, this gives, for
$\Phi,\Psi\in\mathfrak C$,
\[
K^{\Phi\Psi}=K^{\Phi}\!\cdot K^{\Psi},
\qquad
K^{\Phi^{*}}_{\omega}(x,y)=\big(K^{\Phi}_\omega(y,x)\big)^{*}.
\]
For dominated kernels the product sum even converges absolutely: if
$k,k'$ dominate $K,K'\in\mathcal K^\nu_{\mathrm{cov}}(L)$ then,
$z\mapsto a(z-y)$ being injective on $\omega$
(Lemma~\ref{lem:hullcompat} and Proposition~\ref{prop:lagarias}),
\[
\sum_{z\in\omega}\|K_\omega(x,z)\|\,\|K'_\omega(z,y)\|
\ \le\ \sum_{m\in\Z^N}k\big(a(x-y)-m\big)\,k'(m)
\ =\ (k\ast k')\big(a(x-y)\big)\ <\ \infty .
\]
Finally, $\Phi\mapsto K^\Phi$ is injective on $\mathfrak C$: every
arrow $(\eta,z)\in\GL$ equals $(\omega-x,\,y-x)$ for $\omega=\eta$,
$x=0$, $y=z$, so $K^\Phi=0$ forces $j(\Phi)=0$, and $j$ is injective
(Step~5).

\textbf{Step 3: $K^\Phi\in\mathcal K^\nu_{\mathrm{cov}}(L)$, and the
isometry.} Let $\Phi\in\ell^{1,\nu}(\mathfrak C)$. The series
$\Psi:=\sum_g\widetilde P_g(\Phi)$ converges absolutely in
$\mathfrak C$, and $\widetilde P_h(\Phi-\Psi)=0$ for every $h$
(Step~4 of the proof of Theorem~\ref{thm:grading}); hence
$j(\Phi-\Psi)\,\mathbf 1_{c^{-1}(h)}=j\big(\widetilde P_h(\Phi-\Psi)\big)=0$
for every $h$, so $j(\Phi-\Psi)=0$ and $\Phi=\Psi$ (Step~5). Since $j$
is linear and contractive,
$j(\Phi)=\sum_g j(\widetilde P_g\Phi)$ uniformly on $\GL$; each
summand is continuous, so $j(\Phi)$ is continuous. By
Lemma~\ref{lem:transport}(i), $K^\Phi$ satisfies the covariance (a)
and $\Phi_{K^\Phi}=j(\Phi)$; by Lemma~\ref{lem:transport}(ii), $K^\Phi$
satisfies the local dependence (b). Domination: for $x,y\in\omega$ the
arrow $(\omega-x,\,y-x)$ has degree $a(y-x)=-a(x-y)$
(Remark~\ref{rem:symbols}(a)), so (i) gives
\[
\|K^{\Phi}_\omega(x,y)\|\ \le\ \big\|\widetilde P_{a(y-x)}(\Phi)\big\|
\ =\ k_0\big(a(x-y)\big),
\qquad k_0(g):=\|\widetilde P_{-g}(\Phi)\|,
\]
and $\|k_0\|_{\ell^{1,\nu}}
=\sum_g\nu(g)\,\|\widetilde P_{-g}(\Phi)\|=\|\Phi\|_{1,\nu}$ by
symmetry of $\nu$ (Definition~\ref{def:weight}). Hence
$K^\Phi\in\mathcal K^\nu_{\mathrm{cov}}(L)$ with
$\|K^\Phi\|_{\mathcal K^\nu}\le\|\Phi\|_{1,\nu}$. Conversely, if $k$ is
any admissible dominator of $K^\Phi$, then for every $h$, by (i) and
$\Phi(\eta,x_h)=K^\Phi_\eta(0,x_h)$,
\[
\|\widetilde P_h(\Phi)\|
=\sup_{\eta\in\dom_h}\|K^{\Phi}_\eta(0,x_h)\|
\ \le\ k\big(a(-x_h)\big)=k(-h),
\]
so $\|k\|_{\ell^{1,\nu}}=\sum_g\nu(g)\,k(g)
\ge\sum_g\nu(g)\,\|\widetilde P_{-g}(\Phi)\|=\|\Phi\|_{1,\nu}$. Thus
$\|K^\Phi\|_{\mathcal K^\nu}=\|\Phi\|_{1,\nu}$, the infimum being
attained at $k_0$.

\textbf{Step 4: surjectivity and conclusion.} Let
$K\in\mathcal K^\nu_{\mathrm{cov}}(L)$ with admissible dominator $k$.
By (b) and Lemma~\ref{lem:transport}(ii), $\Phi_K$ is continuous on
$\GL$; put $f_g:=\Phi_K\,\mathbf 1_{c^{-1}(g)}$, a continuous function
supported on the compact open set $c^{-1}(g)$, i.e.\
$f_g\in\mathfrak C_g$, with
\[
\|f_g\|_{\mathfrak C}
=\sup_{\omega\in\dom_g}\|K_\omega(0,x_g)\|
\ \le\ k\big(a(-x_g)\big)=k(-g)
\]
(Step~2 of the proof of Theorem~\ref{thm:grading} for the equality).
Then $\Phi:=\sum_g f_g$ converges absolutely in $\mathfrak C$
($\sum_g\|f_g\|\le\|k\|_{\ell^1}<\infty$),
$\widetilde P_h(\Phi)=f_h$ for every $h$ (continuity of
$\widetilde P_h$ and Step~4 there), and
\[
\sum_h\nu(h)\,\|\widetilde P_h(\Phi)\|
\ \le\ \sum_h\nu(h)\,k(-h)=\|k\|_{\ell^{1,\nu}}<\infty
\]
by symmetry of $\nu$; so $\Phi\in\ell^{1,\nu}(\mathfrak C)$. Moreover
$j(\Phi)=\sum_g f_g=\Phi_K$ (uniform convergence; $j$ is the identity
on $C_\co$), whence $K^{\Phi}=K^{\Phi_K}=K$ by
Lemma~\ref{lem:transport}(i). The map is onto; by Step~2 the kernel
product and involution of members of
$\mathcal K^\nu_{\mathrm{cov}}(L)$ are again members and correspond to
the operations of $\ell^{1,\nu}(\mathfrak C)$, so the Banach
$*$-algebra axioms transport. Unitality: $\mathbf 1=u_0$, and
$(\omega-x,\,y-x)$ is a unit iff $x=y$, so
$K^{\mathbf 1}_\omega(x,y)=\delta_{x,y}\mathbf 1$, of norm
$\nu(0)=1$.
\end{proof}

\begin{corollary}[The kernel classes are Euclidean]\label{cor:weightclasses}
\emph{(i) Doubling weights.} Let $w:[0,\infty)\to[1,\infty)$ be
nondecreasing with $w(2t)\le D\,w(t)$. Assume also that $w(0)=1$
and $w(u+v)\le w(u)w(v)$ for $u,v\ge0$, and set
$\nu_w(g):=w(|g|)$ on $\Z^N$. Let $\mathcal E^{w}_{\mathrm{cov}}(L)$ be the class obtained
from Definition~\ref{def:cd} by replacing the dominating condition with
\[
\|K_\omega(x,y)\|\ \le\ h(x-y),\qquad
\|h\|_{w}:=\sum_{z\in L-L}h(z)\,w(|z|)<\infty ,
\]
for some $h:L-L\to[0,\infty)$,
normed by $\|K\|_{\mathcal E^{w}}:=\inf\{\|h\|_{w}\}$.
Then
\[
\mathcal K^{\nu_w}_{\mathrm{cov}}(L)\ =\ \mathcal E^{w}_{\mathrm{cov}}(L)
\]
with equivalent norms, the constants depending only on $D$ and $c_\pm$.
In particular, for $w(t)=(1+t)^s$, this is exactly the class
of hull-covariant kernels admitting a common Euclidean
dominating profile with finite weighted $\ell^1$-moment
of order $s$.
This is stronger than a pointwise bound of order $s$ alone.

\emph{(ii) Sub-exponential scale.} For $0<\alpha<1$, $\lambda>0$, put
$\nu_\lambda(g):=e^{\lambda|g|^\alpha}$ and let
$\mathcal E^{\alpha,\lambda}_{\mathrm{cov}}(L)$ be defined as above with
$w(t)=e^{\lambda t^\alpha}$. Then
\[
\mathcal E^{\alpha,\lambda}_{\mathrm{cov}}
\ \subseteq\ \mathcal K^{\nu_{\lambda/c_+^{\alpha}}}_{\mathrm{cov}}
\ \subseteq\ \mathcal E^{\alpha,\;\lambda\,(c_-/c_+)^{\alpha}}_{\mathrm{cov}} ,
\]
so the decay exponent $\alpha$ is preserved and only the rate moves,
within the window $\big[(c_-/c_+)^{\alpha}\lambda,\ \lambda\big]$. In
particular
$\bigcup_{\lambda>0}\mathcal K^{\nu_\lambda}_{\mathrm{cov}}
 =\bigcup_{\lambda>0}\mathcal E^{\alpha,\lambda}_{\mathrm{cov}}$
(and likewise for intersections): the sub-exponential decay classes of
exponent $\alpha$ over the hull do not depend on the choice of metric.
The admissibility hypotheses of Theorem~\ref{thm:main}(ii) are
conditions on the weight over $\Z^N$ and are unaffected by this
translation.
\end{corollary}

\begin{proof}
Conditions (a) and (b) of Definition~\ref{def:cd} are common to both
classes and are untouched throughout: only the dominating profiles are
exchanged. Recall that $x-y\in\omega-\omega\subseteq L-L$ for all
$\omega\in\trans$ and $x,y\in\omega$ (Lemma~\ref{lem:hullcompat}), so
Euclidean profiles need only be defined on $L-L$; and $a$ restricts to
a bijection of $L-L$ onto $a(L-L)\subseteq\Z^N$
(Proposition~\ref{prop:lagarias}).

\emph{Exchange of profiles.} Given $h$, set $k:=h\circ a^{-1}$ on
$a(L-L)$ and $k:=0$ elsewhere; given $k$, set $h:=k\circ a$ on $L-L$.
In both directions domination is preserved \emph{exactly}: for
$x,y\in\omega$,
\[
\|K_\omega(x,y)\|\ \le\ h(x-y)\ =\ k\big(a(x-y)\big),
\]
since $x-y\in L-L$.

\emph{(i).} For $c>0$, monotonicity and iterated doubling give
$w(ct)\le D^{\,\lceil\log_2c\rceil\vee0}\,w(t)$ for all $t\ge0$.
Hence, by \eqref{eq:bilip}, for every $z\in L-L$,
\[
w(|a(z)|)\le w(c_+|z|)\le D_+\,w(|z|),
\qquad
w(|z|)\le w\big(c_-^{-1}|a(z)|\big)\le D_-\,w(|a(z)|),
\]
with $D_+:=D^{\lceil\log_2c_+\rceil\vee0}$ and
$D_-:=D^{\lceil\log_2c_-^{-1}\rceil\vee0}$. Therefore, for
$h\mapsto k$,
\[
\|k\|_{\ell^{1,\nu_w}}
=\sum_{z\in L-L}h(z)\,w(|a(z)|)\ \le\ D_+\,\|h\|_{w},
\]
and, for $k\mapsto h$,
\[
\|h\|_{w}=\sum_{z\in L-L}k(a(z))\,w(|z|)
\ \le\ D_-\sum_{z\in L-L}k(a(z))\,w(|a(z)|)
\ \le\ D_-\,\|k\|_{\ell^{1,\nu_w}},
\]
the last step dropping the (nonnegative) terms of
$\|k\|_{\ell^{1,\nu_w}}$ indexed by $g\notin a(L-L)$. Taking infima,
\[
D_-^{-1}\,\|K\|_{\mathcal E^{w}}\ \le\ \|K\|_{\mathcal K^{\nu_w}}
\ \le\ D_+\,\|K\|_{\mathcal E^{w}} ,
\]
and in particular the classes coincide. For $w(t)=(1+t)^{s}$ one may
take $D=2^{s}$, since $1+2t\le2(1+t)$, and $\nu_w(g)=(1+|g|)^{s}$.

\emph{(ii).} Put $\mu:=\lambda/c_+^{\alpha}$ and
$\lambda':=\mu\,c_-^{\alpha}=\lambda\,(c_-/c_+)^{\alpha}$. By
\eqref{eq:bilip}, for $z\in L-L$,
\[
\mu\,|a(z)|^{\alpha}\ \le\ \mu\,c_+^{\alpha}\,|z|^{\alpha}
\ =\ \lambda\,|z|^{\alpha},
\qquad
\lambda'\,|z|^{\alpha}\ \le\ \lambda'\,c_-^{-\alpha}\,|a(z)|^{\alpha}
\ =\ \mu\,|a(z)|^{\alpha} .
\]
Exponentiating and summing shows that the exchange $h\mapsto k$ gives
the first inclusion, with constant $1$:
$\|k\|_{\ell^{1,\nu_\mu}}\le\|h\|_{w_\lambda}$. The exchange
$k\mapsto h$ gives the second inclusion, also with constant $1$:
$\|h\|_{w_{\lambda'}}\le\|k\|_{\ell^{1,\nu_\mu}}$. The identity of the
unions follows because each class of either scale is contained in some
class of the other; for the
intersections, given $K\in\bigcap_\lambda\mathcal K^{\nu_\lambda}$ and
a target rate $\lambda'>0$, the second inclusion with
$\mu:=\lambda'c_-^{-\alpha}$ places $K$ in
$\mathcal E^{\alpha,\lambda'}_{\mathrm{cov}}$, and symmetrically with
the first inclusion.
The bi-Lipschitz comparison alone does not imply equivalence
of the two stretched-exponential weights at the same numerical
rate. The inclusions above compare the two radial scales;
preservation of a fixed physical rate follows instead from
Proposition~\ref{prop:physicalweights}.
\end{proof}

For a hull-covariant kernel satisfying Definition~\ref{def:cd}(a)--(b),
the counting bound of Lemma~\ref{lem:retcount} and
a uniform pointwise estimate
\[
\|K_\omega(x,y)\|\le C(1+|x-y|)^{-q}
\]
implies membership in $\mathcal E^{(1+\cdot)^s}_{\mathrm{cov}}(L)$
whenever $q>s+\kappa$.
Indeed, dyadic shells give a convergent geometric majorant
$\sum_{j\ge0}2^{j(\kappa+s-q)}$ for the weighted profile sum.

\begin{proposition}[Exact physical weights]\label{prop:physicalweights}
Let $L$ be FLC, and let $\Psi:\R^N\to\R^d$ be the linear map
satisfying $\Psi\circ a=\mathrm{id}_{\Gamma_L}$.
For either
\[
w(t)=(1+t)^s,\qquad s\ge0,
\]
or
\[
w(t)=e^{\lambda t^\alpha},
\qquad \lambda>0,\quad 0<\alpha<1,
\]
define
\[
\widetilde\nu_w(g):=w(|\Psi g|),\qquad g\in\Z^N.
\]
Then $\widetilde\nu_w$ is a normalized symmetric
submultiplicative GRS weight, and the dictionary induces an
isometric $*$-isomorphism
\[
\ell^{1,\widetilde\nu_w}(\mathfrak C)
\cong
\mathcal E^w_{\mathrm{cov}}(L).
\]
Consequently $\mathcal E^w_{\mathrm{cov}}(L)$ is spectrally
invariant in $\mathfrak C$ and stable under its holomorphic
functional calculus. In the stretched-exponential case,
the parameter $\lambda$ is preserved exactly.
\end{proposition}

\begin{proof}
Both choices of $w$ are nondecreasing, satisfy $w(0)=1$ and
$w(u+v)\le w(u)w(v)$. Since
$|\Psi(g+h)|\le|\Psi g|+|\Psi h|$, the function
$\widetilde\nu_w$ is a weight. Moreover,
\[
\lim_{m\to\infty}\widetilde\nu_w(mg)^{1/m}
=
\lim_{m\to\infty}w(m|\Psi g|)^{1/m}
=1
\]
for each $g$, for both choices of $w$.

Put $\Sigma=a(L-L)$. A dominating profile on $\Z^N$ may be
set equal to zero outside $\Sigma$ without changing domination
or increasing its norm. For $z\in L-L$,
\[
\widetilde\nu_w(a(z))=w(|z|).
\]
Thus the profile correspondence
\[
k(a(z))=h(z),\qquad k=0\ \text{on }\Z^N\setminus\Sigma,
\]
preserves the weighted sums exactly. Taking infima over
dominating profiles and applying
Proposition~\ref{prop:dictionary} proves the isometric
identification. Spectral invariance and holomorphic stability
follow from Proposition~\ref{prop:weightedSI}.
\end{proof}

\section{Wiener's lemma for quasicrystals}\label{sec:consequences}
For $A=K^{\Phi}\in\mathcal K^\nu_{\mathrm{cov}}(L)$,
$\pi_\omega(A)$ denotes $\pi_\omega(\Phi)$, the bounded operator on
$\ell^2(\omega)\otimes\Cx^n$ whose integral kernel is $A_\omega$
(Proposition~\ref{prop:dictionary}(ii)--(iii)); this is consistent,
the dictionary being bijective.

\begin{corollary}[Hull-uniform Wiener lemma]\label{cor:wiener}
Let $L$ be repetitive and FLC, let $\nu$ be a GRS weight,
and let $A\in\mathcal{K}^\nu_{\mathrm{cov}}(L)$. If $\pi_\omega(A)$ is
invertible in $\Bd\big(\ell^2(\omega)\otimes\Cx^n\big)$ for \emph{one}
$\omega\in\trans$, then $A$ is invertible in
$\mathcal{K}^\nu_{\mathrm{cov}}(L)$. Consequently $\pi_{\omega'}(A)$
is invertible for \emph{every} $\omega'\in\trans$, with
$\sup_{\omega'\in\trans}\big\|\pi_{\omega'}(A)^{-1}\big\|
\le\|A^{-1}\|_{\mathcal K^\nu}$, and $A^{-1}$ obeys the same
off-diagonal decay class as $A$: a single $k\in\ell^{1,\nu}(\Z^N)$
dominates $(A^{-1})_{\omega'}$ simultaneously for all
$\omega'\in\trans$.
\end{corollary}

\begin{proof}
Write $A=K^{\Phi}$ with $\Phi\in\ell^{1,\nu}(\mathfrak C)$
(Proposition~\ref{prop:dictionary}(ii)), so
$\pi_\omega(A)=\pi_\omega(\Phi)$. By Theorem~\ref{thm:main}(iii),
$\pi_\omega$ is faithful and
$\pi_\omega\big[\ell^{1,\nu}(\mathfrak C)\big]$ is inverse-closed in
$\Bd\big(\ell^2(\omega)\otimes\Cx^n\big)$; hence there is
$\Psi\in\ell^{1,\nu}(\mathfrak C)$ with
$\pi_\omega(\Psi)=\pi_\omega(\Phi)^{-1}$. Then
$\pi_\omega(\Phi\Psi-\mathbf 1)=0=\pi_\omega(\Psi\Phi-\mathbf 1)$, and
faithfulness gives $\Phi\Psi=\Psi\Phi=\mathbf 1$: the element $\Phi$
is invertible in $\ell^{1,\nu}(\mathfrak C)$, with
$\Phi^{-1}=\Psi$. Transporting through the isometric isomorphism,
$A^{-1}=K^{\Psi}\in\mathcal K^\nu_{\mathrm{cov}}(L)$. For the
consequences: applying the unital homomorphism $\pi_{\omega'}$ to
$\Phi\Psi=\Psi\Phi=\mathbf 1$ shows
$\pi_{\omega'}(A)^{-1}=\pi_{\omega'}(A^{-1})$, of norm at most
$\|A^{-1}\|_{\mathcal K^\nu}$ by
Proposition~\ref{prop:dictionary}(iii); and
$k_0(g):=\|\widetilde P_{-g}(\Psi)\|_{\mathfrak C}$ dominates
$(A^{-1})_{\omega'}$ for every $\omega'$, by
Proposition~\ref{prop:dictionary}(ii).
\end{proof}

\begin{corollary}[Overlap matrices / LCAO on quasicrystals]\label{cor:overlap}
In the setting of Corollary~\ref{cor:wiener}, let $S=\mathbf 1+A$ with
$A\in\mathcal{K}^\nu_{\mathrm{cov}}(L)$, and suppose
$\pi_\omega(S)\ \ge\ \delta>0$ on $\ell^2(\omega)\otimes\Cx^n$ for
\emph{some} $\omega\in\trans$. Then $S=S^{*}$,
$\pi_{\omega'}(S)\ge\delta$ for \emph{every} $\omega'\in\trans$, and
\[
S^{-1},\ S^{1/2},\ S^{-1/2}\ \in\ \mathcal{K}^\nu_{\mathrm{cov}}(L),
\]
each with hull-uniform decay, and
$\pi_{\omega'}\big(S^{\pm1/2}\big)=\big(\pi_{\omega'}(S)\big)^{\pm1/2}$
for every $\omega'$.
\end{corollary}

\begin{proof}
Write $S=K^{\Sigma}$, $\Sigma=\mathbf 1+\Phi\in\ell^{1,\nu}(\mathfrak C)$.
Since $\pi_\omega(\Sigma)\ge\delta$ is self-adjoint and $\pi_\omega$
is a faithful $*$-homomorphism (Theorem~\ref{thm:main}(iii)),
$\Sigma^{*}=\Sigma$, i.e.\ $S=S^{*}$. The image
$\pi_{\omega}(\mathfrak C)$ is a unital C*-subalgebra of
$\Bd\big(\ell^2(\omega)\otimes\Cx^n\big)$ and $\pi_\omega$ is
isometric, so by spectral permanence of C*-subalgebras
\[
\mathrm{sp}_{\mathfrak C}(\Sigma)
=\mathrm{sp}_{\Bd}\big(\pi_\omega(\Sigma)\big)
\subseteq[\delta,\|\Sigma\|_{\mathfrak C}] .
\]
The same identity at any other $\omega'$ gives
$\mathrm{sp}_{\Bd}(\pi_{\omega'}(\Sigma))\subseteq[\delta,\infty)$
with $\pi_{\omega'}(\Sigma)$ self-adjoint, i.e.\
$\pi_{\omega'}(S)\ge\delta$ for every $\omega'\in\trans$.

By Proposition~\ref{prop:weightedSI}(i),
$\mathrm{sp}_{\ell^{1,\nu}(\mathfrak C)}(\Sigma)
=\mathrm{sp}_{\mathfrak C}(\Sigma)\subseteq[\delta,\|\Sigma\|_{\mathfrak C}]$,
a compact set avoiding the branch cut $(-\infty,0]$. Each of
$f(z)=z^{-1},\,z^{1/2},\,z^{-1/2}$ (principal branches) is holomorphic
on a neighbourhood of it, so stability under the holomorphic
functional calculus, Proposition~\ref{prop:weightedSI}(iii), yields
$\Sigma^{-1},\Sigma^{\pm1/2}\in\ell^{1,\nu}(\mathfrak C)$, defined by
the Riesz--Dunford integral. Any bounded unital homomorphism
intertwines that integral, so
$\pi_{\omega'}(\Sigma^{\pm1/2})=\big(\pi_{\omega'}(\Sigma)\big)^{\pm1/2}$
and $\pi_{\omega'}(\Sigma^{-1})=\pi_{\omega'}(\Sigma)^{-1}$; in
particular $\pi_{\omega'}(\Sigma^{1/2})$ is self-adjoint for every
$\omega'$, and faithfulness gives
$\big(\Sigma^{1/2}\big)^{*}=\Sigma^{1/2}$, so
$\Sigma^{1/2}$ is the (unique) positive square root. Transporting
through Proposition~\ref{prop:dictionary}(ii),
$S^{-1},S^{\pm1/2}\in\mathcal K^\nu_{\mathrm{cov}}(L)$, each dominated
by a single element of $\ell^{1,\nu}(\Z^N)$ uniformly over the hull.
\end{proof}

\begin{remark}[Fermi projectors]
Let $H=H^*\in\mathcal K^\nu_{\mathrm{cov}}(L)$, where $\nu$
is a GRS weight. At positive temperature, the Fermi--Dirac
function is holomorphic on a neighbourhood of the real
spectrum of $H$. Thus holomorphic functional calculus yields
$f_T(H)\in\mathcal K^\nu_{\mathrm{cov}}(L)$, with a
hull-uniform profile and without a spectral-gap assumption.

At zero temperature, suppose the Fermi level lies in a gap
of $\mathrm{sp}_{\mathfrak C}(H)$. The occupied part of the
spectrum is then separated from the remaining spectrum, and
its Riesz projection belongs to the same weighted algebra.
This preserves the available decay class; a spectral gap
alone does not improve it to exponential decay.

Exponential localization requires additional assumptions on
the Hamiltonian, such as appropriate uniform finite-range or
exponentially decaying interactions, together with the
corresponding resolvent estimates; see \cite{BBR,CO,OTC}.
Here are precise perturbation thresholds. In a fixed weighted
algebra let $V=V^*$ and put $\varepsilon=\|V\|_{\mathcal K^\nu}$
(or $\varepsilon=\|V\|_{\mathcal E^w}$ in an exact physical class).
These norms dominate the ambient C*-norm. If $H$ has the gap
$(E_F-\Delta,E_F+\Delta)$ and $\varepsilon<\Delta$, then $H+V$
has the gap
\[
(E_F-(\Delta-\varepsilon),\ E_F+(\Delta-\varepsilon)).
\]
Indeed, for real $t$ in this interval,
$\|(t-H)^{-1}V\|_{\mathfrak C}<1$, and a Neumann series makes
$t-H-V$ invertible. Hence its Fermi projection is in the same
weighted algebra. Likewise, if $S\ge\delta1$ and
$\varepsilon<\delta$, then $S+V\ge(\delta-\varepsilon)1$, so the
overlap conclusions persist. At positive temperature no smallness
condition is needed for membership: every self-adjoint $H+V$ in the
fixed weighted algebra admits the same Fermi--Dirac functional
calculus. These are statements within that algebra, not assertions
that a gap improves the decay rate.
\end{remark}

\begin{remark}[Magnetic twist]\label{rem:magnetic}
Represent a constant magnetic field by a real skew-symmetric
bilinear form $B$ on $\R^d$, and set
\[
\sigma_B(g,h)
=
\exp\left(\frac{i}{2}B(\Psi g,\Psi h)\right).
\]
This is a normalized $\T$-valued group $2$-cocycle.
Its pullback through $c$ defines the twisted groupoid algebra,
equivalently the cocycle deformation
$\mathfrak C\{\sigma_B\}$ of \cite[Rem.~2.7]{JM}.

With the convention that twisted convolution includes the
factor $\sigma_B(c(\alpha),c(\beta))$ for $\alpha\beta=\gamma$,
the regular-representation kernel is
\[
K^{\Phi,B}_\omega(x,y)
=
e^{\frac{i}{2}B(x,y)}
\Phi(\omega-x,y-x).
\]
Indeed, the phase in the matrix coefficient is
$\sigma_B(a(y-x),-a(y))=e^{\frac{i}{2}B(x,y)}$.
Consequently
\[
K^{\Phi,B}_{\omega-t}(x-t,y-t)
=
e^{\frac{i}{2}B(t,x-y)}
K^{\Phi,B}_\omega(x,y).
\]
This is magnetic covariance, not the ordinary covariance
of Definition~\ref{def:cd}.

The phase has modulus one, so dominating profiles and the
bisection norm estimates are unchanged. The deformed algebra
retains its topological $\Z^N$-grading, and
Proposition~\ref{prop:weightedSI} applies to every GRS weight.
Thus the Wiener conclusion extends to the magnetic kernel
class with the covariance displayed above. More explicitly, removing
the phase $e^{\frac i2B(x,y)}$ converts magnetic covariance to the
ordinary covariance of Lemma~\ref{lem:transport}; this proves the
bijection and preserves the least-profile norm. Ordinary composition
of the magnetic matrices corresponds to twisted convolution.
The reduced twisted algebra has a faithful unit-space expectation:
the sum-of-squares argument of Step~5 of Theorem~\ref{thm:grading}
is unchanged, because the cocycle is normalized and
$\sigma_B(g,-g)=1$. Its diagonal matrix coefficients are again
the unit-space expectation evaluated at the ranges of the fibre
arrows. The proof of Lemma~\ref{lem:densefaithful} therefore applies,
so the one-configuration conclusion still requires a dense orbit
(or repetitivity of the hull).
\end{remark}

\begin{remark}[What this does \emph{not} give]
No fixed exponential weights (Remark~\ref{rem:weights}); no rates
for the SCF/nonlinear problem; and no infinite-local-complexity
tilings, since the framework assumes FLC with respect to translations.
\end{remark}

\subsection{Norm-controlled inversion: the quantitative regime}

Theorem~\ref{thm:main} and Proposition~\ref{prop:weightedSI} assert
\emph{membership}: the inverse of an element of
$\ell^{1,\nu}(\mathfrak{C})$ again lies in
$\ell^{1,\nu}(\mathfrak{C})$. They give no bound on
$\|\Phi^{-1}\|_{\ell^{1,\nu}}$ in terms of measurable data. Following
\cite{GK2}, the algebra $\ell^{1,\nu}(\mathfrak{C})$ is said to admit
\emph{norm-controlled inversion} in $\mathfrak{C}$ if there is a
function $h:(0,\infty)^2\to(0,\infty)$, nondecreasing in each
variable, such that
$\|\Phi^{-1}\|_{\ell^{1,\nu}}\le
h\big(\|\Phi\|_{\ell^{1,\nu}},\|\Phi^{-1}\|_{\mathfrak{C}}\big)$
whenever $\Phi\in\ell^{1,\nu}(\mathfrak{C})$ is invertible in
$\mathfrak{C}$. The scalar history is instructive: for the unweighted
Wiener algebra $\ell^{1}(\Z^d)\subset C(\mathbb{T}^d)$ norm control
\emph{fails} (Nikolski \cite{Nik}; see \cite[\S1]{SS}), while polynomial weights of positive order restore it in the
classical settings \cite{ENZ,GK1,GK2,GKII,SS}.
Flores's cross-section theorem \cite[Theorem~1.4]{Flores25} also
applies to the address Fell bundle with
$\nu_s(g)=(1+|g|)^s$: choose $p>N/s$ so that
$\nu_s^{-1}\in\ell^p(\Z^N)$.
Thus existence of polynomial norm control in this graded algebra
has a direct antecedent. The purpose of the estimates below is to
track the supported return-vector growth through
$\theta=\kappa/(\kappa+s)$ and to obtain explicit control of
hull-uniform profiles, including exact physical
stretched-exponential weights. The unweighted examples transfer
Nikolski's obstruction rather than introduce a new endpoint mechanism.

Throughout, $\nu_s(g):=(1+|g|)^s$ on $\Z^N$ for $s>0$; the
standing inequalities are
$\|\Phi\|_{\mathfrak{C}}\le\|\Phi\|_{\ell^{1}}\le\|\Phi\|_{\ell^{1,\nu_s}}$.

\begin{lemma}[Growth of the return vectors]\label{lem:retcount}
For $\rho\ge0$ put
$\mathcal{N}_L(\rho):=\#\big((L-L)\cap\bar B_\rho\big)$, the ball
$\bar B_\rho\subset\R^d$. Then:
\begin{enumerate}
\item[(i)] For every FLC Delone set,
$\mathcal{N}_L(\rho)\le 3^N\max(1,c_+)^N\,(1+\rho)^N$, with $c_+$ from
Lemma~\ref{lem:addcomp}.
\item[(ii)] If $L$ is a Meyer set, so that $L-L$ is uniformly discrete
with separation $r'>0$, then
$\mathcal{N}_L(\rho)\le(1+2/r')^{d}\,(1+\rho)^{d}$.
\item[(iii)] For every $(r,R)$-Delone set,
$\mathcal{N}_L(\rho)\ge c_d\,(\rho/R)^{d}$ for $\rho\ge R$.
\end{enumerate}
Fix $\kappa$ and $C_\kappa$ once and for all so that
$\mathcal{N}_L(\rho)\le C_\kappa(1+\rho)^{\kappa}$; by (i)--(iii) this
is always possible with $d\le\kappa\le N$, and $\kappa=d$ when $L$ is
Meyer.
\end{lemma}

\begin{proof}
(i) The address map is injective on $\Gamma_L\supseteq L-L$ and
$|a(z)|\le c_+|z|$ by \eqref{eq:bilip}, so
$\mathcal{N}_L(\rho)\le\#\big(\Z^N\cap\bar B_{c_+\rho}\big)
\le(2c_+\rho+1)^N$. (ii) The points of $(L-L)\cap\bar B_\rho$ are
$r'$-separated, so the balls of radius $r'/2$ centred at them are
disjoint and contained in $\bar B_{\rho+r'/2}$; compare volumes.
(iii) Fix $x_0\in L$; then
$(L\cap\bar B_\rho(x_0))-x_0\subseteq(L-L)\cap\bar B_\rho$, and, for $\rho\ge2R$, the $R$-balls centred at
$L\cap\bar B_\rho(x_0)$ cover $\bar B_{\rho-R}(x_0)$.
Thus the number of centres is at least
$((\rho-R)/R)^d\ge(\rho/(2R))^d$.
For $R\le\rho<2R$, the point $x_0$ gives the same bound.
One may therefore take $c_d=2^{-d}$.
\end{proof}

\begin{remark}
If $L$ is linearly repetitive with repetitivity function
$M(\rho)\le C_{\mathrm{LR}}\,\rho$, then every return vector of length
$\le\rho$ is realized inside $\bar B_{M(\rho)+\rho}(x_0)$, whence
$\mathcal{N}_L(\rho)\le\big[\#(L\cap\bar B_{M(\rho)+\rho}(x_0))\big]^2
\lesssim\rho^{2d}$: one may take $\kappa\le 2d$. Note also that (i)
answers a natural question: for FLC sets the growth of the return
vectors is \emph{always} at most polynomial, of order at most the rank
$N$.
\end{remark}

\begin{lemma}[Interpolation of the $\ell^1$-norm]\label{lem:interp}
Set $\theta:=\dfrac{\kappa}{\kappa+s}\in(0,1)$. There is a constant
$C_\kappa':=C_\kappa\max(1,c_-^{-1})^{\kappa}$ such that, for every
$\Phi\in\ell^{1,\nu_s}(\mathfrak{C})$,
\begin{equation}\label{eq:interp}
\|\Phi\|_{\ell^{1}}\ \le\ (C_\kappa'+1)\,
\|\Phi\|_{\ell^{1,\nu_s}}^{\,\theta}\,
\|\Phi\|_{\mathfrak{C}}^{\,1-\theta}.
\end{equation}
\end{lemma}

\begin{proof}
\emph{Support.} $\Phi_g\ne0$ forces $c^{-1}(g)\ne\emptyset$, i.e.\
$g\in S:=c(\GL)\subseteq a(L-L)$ (Lemma~\ref{lem:hullcompat}). By
\eqref{eq:bilip}, $|a(z)|\le R$ implies $|z|\le R/c_-$, so
$\#(S\cap\bar B_R)\le\mathcal{N}_L(R/c_-)\le C_\kappa'(1+R)^{\kappa}$.

\emph{Coefficient bound.} $\|\Phi_g\|_{\mathfrak{C}}\le
\|\iota\Phi\|_{\mathfrak{C}}$ for every $g$: the Fourier-coefficient
maps of a topologically graded C*-algebra are contractive \cite{Exel}. In
the present setting makes this immediate: for $b\in C^*_r(\GL)$ the
continuous symbol $j(b)\in C_0(\GL)$ satisfies
$\|j(b)\|_\infty\le\|b\|$, the restriction $j(b)\,1_{c^{-1}(g)}$
represents $\widetilde P_g(b)$, and for a function supported on a bisection the
reduced norm equals the sup norm.

\emph{Split and optimize.} The assertion is immediate for $\Phi=0$,
so assume $\Phi\ne0$. For any $R\ge0$,
\[
\|\Phi\|_{\ell^{1}}
=\sum_{g\in S,\,|g|\le R}\|\Phi_g\|+\sum_{|g|>R}\|\Phi_g\|
\ \le\ C_\kappa'(1+R)^{\kappa}\,\|\Phi\|_{\mathfrak{C}}
+(1+R)^{-s}\,\|\Phi\|_{\ell^{1,\nu_s}},
\]
using $\nu_s(g)\ge(1+R)^{s}$ for $|g|>R$ in the second sum. Setting
$1+R=t^{1/(\kappa+s)}$ with
$t:=\|\Phi\|_{\ell^{1,\nu_s}}/\|\Phi\|_{\mathfrak{C}}\ge1$ makes both
terms equal to
$\|\Phi\|_{\ell^{1,\nu_s}}^{\kappa/(\kappa+s)}
\|\Phi\|_{\mathfrak{C}}^{s/(\kappa+s)}$ up to the constants shown.
\end{proof}

\begin{proposition}[Differential subalgebra of explicit order]\label{prop:diffsub}
Let $s>0$, $\theta=\kappa/(\kappa+s)$, and
$C_s:=\max(1,2^{s-1})$. Then, for all
$\Phi,\Psi\in\ell^{1,\nu_s}(\mathfrak{C})$,
\begin{equation}\label{eq:leibniz}
\|\Phi\Psi\|_{\ell^{1,\nu_s}}\ \le\
C_s\big(\|\Phi\|_{\ell^{1,\nu_s}}\|\Psi\|_{\ell^{1}}
+\|\Phi\|_{\ell^{1}}\|\Psi\|_{\ell^{1,\nu_s}}\big),
\end{equation}
and consequently, with $D:=2C_s(C_\kappa'+1)$,
\begin{equation}\label{eq:brandenburg}
\|\Phi^2\|_{\ell^{1,\nu_s}}\ \le\
D\,\|\Phi\|_{\ell^{1,\nu_s}}^{\,1+\theta}\,
\|\Phi\|_{\mathfrak{C}}^{\,1-\theta}.
\end{equation}
Thus $\ell^{1,\nu_s}(\mathfrak{C})$ satisfies the hypothesis
{\rm(2.2)} of \cite{SS} with exponent $\theta$; equivalently, it is a
differential $*$-subalgebra of $\mathfrak{C}$ of order
$1-\theta=s/(\kappa+s)$ in the product-form convention of
\cite[(2.1)]{ShSu}.
\end{proposition}

\begin{proof}
Since $1+|g|\le(1+|h|)+(1+|g-h|)$ and $t\mapsto t^{s}$ is subadditive
for $s\le1$ and convex for $s\ge1$, one has
$\nu_s(g)\le C_s\big(\nu_s(h)+\nu_s(g-h)\big)$. With
$\phi(g):=\|\Phi_g\|$, $\psi(g):=\|\Psi_g\|$ and the coordinate
majorization $\|(\Phi\Psi)_g\|\le(\phi*\psi)(g)$,
\[
\|\Phi\Psi\|_{\ell^{1,\nu_s}}
\le\sum_{g,h}\nu_s(g)\,\phi(h)\psi(g-h)
\le C_s\sum_{g,h}\big(\nu_s(h)+\nu_s(g-h)\big)\phi(h)\psi(g-h),
\]
which is \eqref{eq:leibniz} by Tonelli. Taking $\Psi=\Phi$ and
inserting \eqref{eq:interp} yields \eqref{eq:brandenburg}; inserting
\eqref{eq:interp} into each factor of \eqref{eq:leibniz} instead gives
the product form of order $1-\theta$.
\end{proof}

\begin{theorem}[Norm-controlled inversion for polynomial weights]\label{thm:nci}
Let $L$ be an FLC Delone set, $s>0$, and $\theta=\kappa/(\kappa+s)$ as
above. Then $\ell^{1,\nu_s}(\mathfrak{C})$ admits norm-controlled
inversion in $\mathfrak{C}=C^*_r(\GL)$: it is inverse-closed, and
whenever $\Phi\in\ell^{1,\nu_s}(\mathfrak{C})$ is invertible in
$\mathfrak{C}$, the explicit product bound \cite[Prop.~2.2, (2.3)]{SS}
holds with constants $D$ and $\theta$ from
Proposition~\ref{prop:diffsub}. In asymptotic form
\cite[Prop.~2.3]{SS}: with $\gamma:=\log_2(1+\theta)\in(0,1)$ and
$\mathrm{cn}(\Phi):=\|\Phi\|_{\ell^{1,\nu_s}}
\|\Phi^{-1}\|_{\mathfrak{C}}$, there are constants $C_1,C_2$ depending
only on $(D,\theta)$ such that, whenever $\mathrm{cn}(\Phi)\ge2$,
\begin{equation}\label{eq:ncibound}
\|\Phi^{-1}\|_{\ell^{1,\nu_s}}\ \le\
C_1\,\|\Phi\|_{\ell^{1,\nu_s}}\,\|\Phi^{-1}\|_{\mathfrak{C}}^{2}\,
\exp\Big(C_2\,\mathrm{cn}(\Phi)^{\frac{2\gamma}{1-\gamma}}
\big(\ln\mathrm{cn}(\Phi)\big)^{\frac{2-\gamma}{1-\gamma}}\Big).
\end{equation}
Moreover, in the repetitive case every $\pi_\omega$ is faithful
by Lemma~\ref{lem:densefaithful}, hence isometric, and
\eqref{eq:ncibound} holds with
$\|\pi_\omega(\Phi)^{-1}\|_{B(\ell^2(\omega)\otimes\Cx^n)}$ in place of
$\|\Phi^{-1}\|_{\mathfrak{C}}$, with the \emph{same} control function
for every configuration $\omega$ simultaneously. More generally this
holds for any configuration with dense orbit. All constants are effective: $D=2C_s(C_\kappa'+1)$ and
$\theta=\kappa/(\kappa+s)$ are computed from
$(s,\kappa,C_\kappa,c_-)$, and for Meyer sets $\kappa=d$ with
$C_\kappa$ read off the separation of $L-L$
(cf.\ Remark~\ref{rem:constants}).
\end{theorem}

\begin{proof}
Proposition~\ref{prop:diffsub} verifies hypothesis (2.2) of \cite{SS} for
the pair $\ell^{1,\nu_s}(\mathfrak{C})\subset\mathfrak{C}$ (Banach
$*$-algebra and C*-algebra with common unit,
$\|\cdot\|_{\mathfrak{C}}\le\|\cdot\|_{\ell^{1,\nu_s}}$), so
\cite[Prop.~2.2]{SS} gives inverse-closedness and the product bound, and
\cite[Prop.~2.3]{SS} gives \eqref{eq:ncibound}. The hull statement follows
because a faithful representation of a C*-algebra is isometric.
\end{proof}

\begin{remark}[A nondecreasing control function]
\label{rem:monotonecontrol}
For fixed differential parameters $D\ge1$ and $0<\theta<1$,
one may use
\[
h(A,B)
=
AB^2
\prod_{k=0}^{\infty}
\left[
1+
(2T^2)^{p_k}
D^{(p_k-1)/\theta}
(1-T^{-2})^{2^k-p_k}
\right],
\qquad
T:=\max\{2,AB\}.
\]
Here $p_k:=(1+\theta)^k$ and $A,B>0$. Since $2^k-p_k\ge0$, every factor is nondecreasing
in $T$, and hence $h$ is nondecreasing in both variables. The logarithm
of the term added to $1$ is
$2^k\log(1-T^{-2})+O((1+\theta)^k)$, so the product converges
locally uniformly for finite $T\ge2$.

The bound follows from \cite[Prop.~2.2]{SS}. A direct verification,
which also handles the case of C*-condition number one, follows.
Let $\mathcal A\subseteq\mathfrak C$ be the weighted algebra, with
isometric involution, common normalized unit, and
\[
\|b^2\|_{\mathcal A}\le
D\|b\|_{\mathcal A}^{1+\theta}\|b\|_{\mathfrak C}^{1-\theta}.
\]
If $\|\Phi\|_{\mathcal A}\le A$ and
$\|\Phi^{-1}\|_{\mathfrak C}\le B$, put
$c=1-\Phi^*\Phi/\|\Phi\|_{\mathfrak C}^2$ and
$\delta_T=1-T^{-2}\in(0,1)$. Positivity in the C*-algebra gives
\[
\|c\|_{\mathfrak C}
=1-\frac1{\|\Phi\|_{\mathfrak C}^2\|\Phi^{-1}\|_{\mathfrak C}^2}
\le\delta_T,
\qquad
\|c\|_{\mathcal A}\le2T^2.
\]
Applying the differential inequality successively to $c^{2^k}$ gives
\[
\|c^{2^k}\|_{\mathcal A}\le
(2T^2)^{p_k}D^{(p_k-1)/\theta}\delta_T^{\,2^k-p_k}
=:u_k.
\]
This is proved by induction, using
$\|c^{2^k}\|_{\mathfrak C}\le\delta_T^{2^k}$.
Binary expansion of $m$ and submultiplicativity then imply
\[
\sum_{m=0}^{\infty}\|c^m\|_{\mathcal A}
\le\prod_{k=0}^{\infty}(1+u_k)<\infty.
\]
The resulting Neumann series converges in $\mathcal A$ and
\[
\Phi^{-1}=\frac1{\|\Phi\|_{\mathfrak C}^2}
\left(\sum_{m=0}^{\infty}c^m\right)\Phi^*.
\]
Since $\|\Phi\|_{\mathfrak C}\ge B^{-1}$, its norm is bounded by
$AB^2\prod_k(1+u_k)=h(A,B)$. In particular, the product bound does
not rely on dividing by $\|c\|_{\mathfrak C}$, which may be zero.

For the physical polynomial norm, the constants $D_\pm$ in the proof
of Corollary~\ref{cor:weightclasses} give the nondecreasing control
\[
h_{L,s}(A,B):=D_-\,h(D_+A,B).
\]
Indeed, $\|K\|_{\mathcal K^{\nu_s}}\le D_+\|K\|_{\mathcal E^w}$
and $\|K^{-1}\|_{\mathcal E^w}\le D_-\|K^{-1}\|_{\mathcal K^{\nu_s}}$.
\end{remark}

\begin{remark}
Theorem~\ref{thm:nci} independently reproves the inverse-closedness of
$\ell^{1,\nu_s}(\mathfrak{C})$ (a special case of
Proposition~\ref{prop:weightedSI}). The two arguments are logically
independent. The GRS--Hulanicki argument in
Proposition~\ref{prop:weightedSI} covers all GRS weights, whereas the
Brandenburg-type differential estimate in Theorem~\ref{thm:nci}
provides quantitative bounds.
\end{remark}

The next proof adapts a weighted product estimate and H\"older
interpolation, as used in \cite{SS} and, for groupoid fibre norms,
\cite[Lemmas~4.5--4.6]{AOP26}. Here both the norm and the summability
constant are formed from degreewise suprema over the hull. In
particular, the physical-length sum is over $\Sigma=a(L-L)$, not
over all of $\Z^N$, where the pulled-back length need not be proper.

\begin{theorem}[Norm control for stretched-exponential weights]
\label{thm:stretchednci}
Let $L$ be FLC and choose either
\[
\ell(g)=|g|
\qquad\text{or}\qquad
\ell(g)=|\Psi g|.
\]
Fix $\lambda>0$ and $0<\alpha<1$, and put
\[
\nu(g)=e^{\lambda\ell(g)^\alpha},
\qquad
q=2^\alpha-1.
\]
For any $\beta\in(q,1)$, define
\[
Q=
\left(
\sum_{g\in\Sigma}
\exp\left[
-\frac{\beta-q}{1-\beta}\lambda\ell(g)^\alpha
\right]
\right)^{1-\beta},
\qquad
\Sigma=a(L-L).
\]
Then $Q<\infty$, and
\begin{equation}\label{eq:stretcheddifferential}
\|\Phi^2\|_{\ell^{1,\nu}}
\le
2Q\,
\|\Phi\|_{\ell^{1,\nu}}^{1+\beta}
\|\Phi\|_{\mathfrak C}^{1-\beta}
\qquad
(\Phi\in\ell^{1,\nu}(\mathfrak C)).
\end{equation}
Consequently $\ell^{1,\nu}(\mathfrak C)$ admits norm-controlled
inversion in $\mathfrak C$, with the control function of
Remark~\ref{rem:monotonecontrol} using $D=2Q$ and $\theta=\beta$.

For $\ell(g)=|\Psi g|$, this gives norm-controlled inversion
in the physical class
$\mathcal E^{\alpha,\lambda}_{\mathrm{cov}}(L)$,
without changing $\lambda$.
\end{theorem}

\begin{proof}
Both choices of $\ell$ are symmetric and subadditive.
For $u\ge v\ge0$,
\[
(u+v)^\alpha\le u^\alpha+qv^\alpha.
\]
Indeed, for $v>0$ the function
$t\mapsto(t+1)^\alpha-t^\alpha$ is decreasing on $[1,\infty)$
and has value $q$ at $1$; the case $v=0$ is immediate.

It follows, by separating the cases $\ell(g)\ge\ell(h)$ and
$\ell(h)\ge\ell(g)$, that
\[
\nu(g+h)
\le
\nu(g)\nu(h)^q+\nu(g)^q\nu(h).
\]
In this proof write
\[
\|\Phi\|_{(t)}
:=
\sum_{g\in\Sigma}
\|\widetilde P_g(\Phi)\|_{\mathfrak C}
e^{t\lambda\ell(g)^\alpha},
\qquad 0\le t\le1.
\]
Extend all coefficient families by zero outside $\Sigma$.
Although $\Sigma$ need not be a subgroup, coefficient majorization
can be summed over $\Z^N\times\Z^N$; Tonelli's theorem then gives
\[
\|\Phi\Theta\|_{(1)}
\le
\|\Phi\|_{(1)}\|\Theta\|_{(q)}
+
\|\Phi\|_{(q)}\|\Theta\|_{(1)}.
\]

Set $b_g=\|\widetilde P_g(\Phi)\|_{\mathfrak C}$.
Contractivity of the coefficient maps gives
$b_g\le\|\Phi\|_{\mathfrak C}$. Hence
\begin{align*}
\|\Phi\|_{(q)}
&=
\sum_{g\in\Sigma}
\big(b_g e^{\lambda\ell(g)^\alpha}\big)^\beta
b_g^{1-\beta}
e^{-(\beta-q)\lambda\ell(g)^\alpha}
\\
&\le
Q\,\|\Phi\|_{(1)}^\beta
\|\Phi\|_{\mathfrak C}^{1-\beta},
\end{align*}
by H\"older's inequality with exponents
$1/\beta$ and $1/(1-\beta)$.

The counting function
$\#\{g\in\Sigma:\ell(g)\le t\}$ has polynomial growth:
for the physical length this is Lemma~\ref{lem:retcount},
and for the address length it follows from that lemma and
\eqref{eq:bilip}. Therefore the sum defining $Q$ converges.

Taking $\Theta=\Phi$ proves
\eqref{eq:stretcheddifferential}.
The pair
$\ell^{1,\nu}(\mathfrak C)\subseteq\mathfrak C$
has a common unit and satisfies the hypotheses of
\cite[Prop.~2.2]{SS}, with differential constant $2Q$
and exponent $\beta$. This yields norm-controlled inversion.
The final assertion follows from
Proposition~\ref{prop:physicalweights}.
\end{proof}

One may take $\beta=2^{\alpha-1}$. Then
\[
Q=
\left(
\sum_{g\in\Sigma}e^{-\lambda\ell(g)^\alpha}
\right)^{1-\beta}.
\]

The constant $Q$ can also be bounded using only counting data.
Write $a_0=(\beta-q)\lambda/(1-\beta)>0$ and let
$N_\ell(t)=\#\{g\in\Sigma:\ell(g)\le t\}\le C_\ell(1+t)^\kappa$,
where $C_\ell=C_\kappa$ for the physical length and
$C_\ell=C_\kappa'$ for the address length. Tonelli's theorem gives
\[
\sum_{g\in\Sigma}e^{-a_0\ell(g)^\alpha}
=\int_0^\infty a_0\alpha t^{\alpha-1}e^{-a_0t^\alpha}N_\ell(t)\dd t.
\]
Using $(1+t)^\kappa\le2^{\max(\kappa-1,0)}(1+t^\kappa)$ yields
\begin{equation}\label{eq:Qcountbound}
Q\le
\left[
C_\ell 2^{\max(\kappa-1,0)}
\left(1+a_0^{-\kappa/\alpha}
\Gamma\left(1+\frac\kappa\alpha\right)\right)
\right]^{1-\beta}.
\end{equation}
Here $\Gamma(\cdot)$ is Euler's gamma function. No summability over
all of $\Z^N$ for the physical length is asserted or required.

\begin{corollary}[Quantitative localization bound for the Fermi projection]\label{cor:ferminci}
Let $H=H^*\in\ell^{1,\nu_s}(\mathfrak{C})$ and suppose
$\mathrm{sp}_{\mathfrak{C}}(H)\cap(E_F-\Delta,E_F+\Delta)=\emptyset$
for some Fermi level $E_F$ and gap half-width $\Delta>0$. Let
$P:=\chi_{(-\infty,E_F]}(H)$. Then $\|P\|_{\ell^{1,\nu_s}}\le B_P$, where
\[
B_P:=
\max\left\{
1,\,
\frac{2\|H\|_{\ell^{1,\nu_s}}+3\Delta}{\pi}\,
h\big(
2\|H\|_{\ell^{1,\nu_s}}+2\Delta,\,
\Delta^{-1}
\big)
\right\}.
\]
Here $h$ is the nondecreasing control function of
Remark~\ref{rem:monotonecontrol}, and consequently, for \emph{every}
$\omega\in\trans$ and all $x,y\in\omega$,
\begin{equation}\label{eq:fermidecay}
\|P_\omega(x,y)\|\ \le\ B_P\,\min(1,c_-)^{-s}\,(1+|x-y|)^{-s}.
\end{equation}
All quantities on the right are explicit in
$\big(\|H\|_{\ell^{1,\nu_s}},\Delta,s,\kappa,C_\kappa,c_\pm\big)$; no
hypothesis is made on the configuration beyond membership in the hull.
\end{corollary}

\begin{proof}
If $P=0$ or $P=\mathbf 1$, the conclusions follow immediately
from $\|\mathbf 1\|_{\ell^{1,\nu_s}}=1$ and $B_P\ge1$.
Otherwise the spectrum has points on both sides of $E_F$,
so $|E_F|\le\|H\|_{\mathfrak C}$. It remains to consider this
nontrivial case.
Let $\Gamma$ be the positively oriented boundary of the rectangle
$[-\|H\|_{\mathfrak{C}}-\Delta,\,E_F]\times[-\Delta,\Delta]$. Every
point of $\Gamma$ is at distance $\ge\Delta$ from
$\mathrm{sp}_{\mathfrak{C}}(H)$, so
$\sup_{\lambda\in\Gamma}\|(\lambda-H)^{-1}\|_{\mathfrak{C}}\le
\Delta^{-1}$; also
$\|\lambda-H\|_{\ell^{1,\nu_s}}\le2\|H\|_{\ell^{1,\nu_s}}+2\Delta$ on
$\Gamma$ (using $|E_F|\le\|H\|_{\mathfrak{C}}$ and $\nu_s(0)=1$), and
$|\Gamma|\le4\|H\|_{\mathfrak{C}}+6\Delta$. Since the gap makes
$\chi_{(-\infty,E_F]}$ agree with a holomorphic function near
$\mathrm{sp}(H)$, the Riesz integral
$P=\frac{1}{2\pi i}\oint_\Gamma(\lambda-H)^{-1}d\lambda$ converges in
$\ell^{1,\nu_s}(\mathfrak{C})$ by Theorem~\ref{thm:nci} and
Proposition~\ref{prop:weightedSI}(iii), giving the bound on
$\|P\|_{\ell^{1,\nu_s}}$. For \eqref{eq:fermidecay}: with
$k(g):=\|\widetilde P_{-g}(P)\|_{\mathfrak{C}}$ one has
$k(g)\le\|P\|_{\ell^{1,\nu_s}}(1+|g|)^{-s}$, and by the dictionary of
\S\ref{sec:kernels} (Def.~\ref{def:cd}, Prop.~\ref{prop:dictionary}),
$\|P_\omega(x,y)\|\le k(a(x-y))$; conclude with
$|a(x-y)|\ge c_-|x-y|$ from Lemma~\ref{lem:addcomp} and
$1+c_-t\ge\min(1,c_-)(1+t)$.
\end{proof}

\begin{proposition}[The weight cannot be removed]\label{prop:nonci}
In the scalar case with $L=\Z^d$, which belongs to the standing
class, one has
$\trans=\{\Z^d\}$, $\GL\cong\Z^d$, $\mathfrak{C}\cong C(\mathbb{T}^d)$
and $\ell^{1}(\mathfrak{C})\cong A(\mathbb{T}^d)$, the classical
Wiener algebra, isometrically and compatibly with the C*-norms. By
Nikolski's theorem \cite[Cor.~3.2.5]{Nik} (in the formulation of
\cite[\S1]{SS}: $\delta_1(\ell^1(G),\widehat G)\ge1/2$ for every
infinite discrete abelian $G$), $A(\mathbb{T}^d)$ does not admit
norm-controlled inversion in $C(\mathbb{T}^d)$. Hence the unweighted
algebra $\ell^{1}(\mathfrak{C})$ does not admit norm-controlled
inversion in $\mathfrak{C}$ in general within the class of FLC Delone
sets, and the weight in Theorem~\ref{thm:nci} cannot be dropped.
\end{proposition}

\begin{proposition}[An aperiodic obstruction]
\label{prop:nonci-aperiodic}
Let $X$ be a nonempty compact Hausdorff space and $\sigma$ a homeomorphism.
For the canonical $\Z$-grading of
$\mathfrak C=C(X)\rtimes_{\sigma,r}\Z$,
the algebra $\ell^1(\mathfrak C)$ does not admit
norm-controlled inversion in $\mathfrak C$.

In particular, the failure occurs for repetitive aperiodic
decorated Delone sets.
\end{proposition}

\begin{proof}
Let $U$ be the canonical unitary of degree $1$.
The gauge action sends $U$ to $zU$ for every $z\in\T$.
Consequently its nonempty spectrum is invariant under all
rotations, and therefore $\mathrm{sp}(U)=\T$.

For $f\in A(\T)$, functional calculus gives
\[
f(U)=\sum_{m\in\Z}\widehat f(m)U^m,
\]
and the canonical grading yields
\[
\|f(U)\|_{\ell^1(\mathfrak C)}
=
\sum_{m\in\Z}|\widehat f(m)|,
\qquad
\|f(U)\|_{\mathfrak C}
=
\|f\|_{C(\T)}.
\]
For invertible $f$, these identities also apply to $f^{-1}$.
Thus norm control in $\ell^1(\mathfrak C)$ would imply norm
control for $A(\T)\subset C(\T)$, contradicting
Proposition~\ref{prop:nonci}.
Taking $X$ to be a minimal aperiodic subshift over a finite alphabet
proves the last assertion. This is an existence statement, not a
classification of all aperiodic hulls.
\end{proof}

\begin{remark}[An undecorated example]\label{rem:nonci-undecorated}
Let $b=(b_m)_{m\in\Z}$ be a uniformly recurrent aperiodic
binary sequence, and set
\[
L=
4\Z\cup(4\Z+1)\cup\{4m+2:b_m=1\}.
\]
This is an FLC Delone set: it is a subset of $\Z$ containing $4\Z$.
Every bounded patch is determined by a finite word of $b$ and its
position within a cell. Uniform recurrence therefore implies
repetitivity of $L$. Moreover $1\in L-L$, so $\Gamma_L=\Z$;
choose the address map $a=\mathrm{id}$. The points of $4\Z$ are recognizable by
the local rule that there is no point at distance $1$ to their
left. Hence any period of $L$ must belong to $4\Z$ and induce
a period of $b$, proving aperiodicity.

Put
\[
Q=\{\omega\in\trans:-1\notin\omega\},
\qquad p=\mathbf 1_Q.
\]
The set $Q$ is nonempty and clopen. If $X_b$ denotes the shift orbit
closure of $b$, the marker rule identifies $Q$ with the configurations
$4\Z\cup(4\Z+1)\cup\{4m+2:b'_m=1\}$ for $b'\in X_b$.
Indeed, approximating a member of $Q$ by anchored translates of $L$
forces the anchor to be a marker. Translation by $4$ thus maps $Q$
onto itself,
and
\[
V=\{(\omega,4):\omega\in Q\}
\]
is a compact open bisection with $r(V)=s(V)=Q$.
Thus $u=\mathbf 1_V$ is a unitary in the corner
$p\mathfrak Cp$, homogeneous of degree $4$.

The gauge action fixes $p$ and sends $u$ to $z^4u$ for $z\in\T$.
Since $z\mapsto z^4$ is onto, the spectrum of $u$ in this corner
is $\T$. Consequently
\[
\|f(u)\|_{\ell^1(\mathfrak C)}=\|f\|_{A(\T)},
\qquad
\|f(u)\|_{\mathfrak C}=\|f\|_{C(\T)}.
\]
For invertible $f\in A(\T)$, let $a_f=f(u)+(1-p)$. Then
\[
a_f^{-1}=f^{-1}(u)+(1-p),
\]
and
\[
\|a_f\|_{\ell^1}\le1+\|f\|_{A(\T)},\qquad
\|a_f^{-1}\|_{\mathfrak C}
=\max\{1,\|f^{-1}\|_{C(\T)}\},\qquad
\|a_f^{-1}\|_{\ell^1}\ge\|f^{-1}\|_{A(\T)}.
\]
Nikolski's obstruction supplies a sequence of invertible functions
$f_j$ with both $\|f_j\|_{A(\T)}$ and $\|f_j^{-1}\|_{C(\T)}$
uniformly bounded, but $\|f_j^{-1}\|_{A(\T)}\to\infty$.
The displayed inequalities and the nondecreasing convention for
control functions transfer this obstruction to $a_{f_j}$.
Thus unweighted norm control fails also for this undecorated
repetitive aperiodic example.
\end{remark}

\begin{remark}[Dichotomy, sharpness, and scope of the bounds]\label{rem:ncidichotomy}
(i) \emph{The framework sees Nikolski.} As $s\to0^+$ one has
$\theta=\kappa/(\kappa+s)\to1$, hence
$\gamma=\log_2(1+\theta)\to1$ and the exponent
$2\gamma/(1-\gamma)$ in \eqref{eq:ncibound} blows up: the control
degenerates exactly at the unweighted endpoint, consistently with
Proposition~\ref{prop:nonci}. Conversely, more decay ($s\to\infty$)
gives $\theta\to0$ and control of shape
$C_1\|\Phi\|\,\|\Phi^{-1}\|_{\mathfrak{C}}^2
e^{C_2(\ln\mathrm{cn})^2}$, the Leibniz-quality regime of \cite{GK2}.
These limiting descriptions concern only the exponents in the
displayed bound. The constants also depend on $s$, and no uniform
limiting estimate as $s\to\infty$ is asserted.
(ii) \emph{Meyer sets and lattice counting.} For Meyer $L$ one may
choose $\kappa=d$, so $\theta=d/(d+s)$ has the same form as in the
corresponding lattice counting argument. This comparison concerns the
differential estimate established here, not optimal control functions
or the best periodic bounds. For general FLC sets the guaranteed choice is
$\kappa=N$; the Meyer condition thus has a quantitative role absent
from the qualitative theory of \S\ref{sec:main}.
(iii) \emph{Stretched-exponential weights.}
Theorem~\ref{thm:stretchednci} proves norm-controlled inversion
using coefficient contractivity and H\"older interpolation.
It does not require interchanging a sum over degrees with
a supremum over configurations. Norm control for more general GRS weights remains a separate question.
(iv) \emph{Matrices.} With the supremum operator norm on each
homogeneous matrix coefficient, all displayed constants are independent
of the finite matrix size. Indeed, coefficient majorization uses
submultiplicativity of the matrix operator norm, and the amplified
gauge-integral formula in Theorem~\ref{thm:grading} gives
$\|\widetilde P_g\otimes\mathrm{id}_{M_n}\|\le1$.
The same counting bounds and \cite[Prop.~2.2]{SS} then apply.
Replacing this norm by an equivalent entrywise norm may introduce
$n$-dependent constants. This observation does not assert an extension
to general Fell-bundle coefficients.
\end{remark}

\printbibliography

@article{FJM,
  author  = {Flores, Felipe and Jaur\'e, Diego and M\u{a}ntoiu, Marius},
  title   = {Symmetry for algebras associated to {F}ell bundles over groups and groupoids},
  journal = {J. Operator Theory},
  volume  = {91},
  year    = {2024},
  pages   = {27--54},
  doi = {10.7900/jot.2022jan27.2382},
}

@article{JM,
  author  = {Jaur\'e, Diego and M\u{a}ntoiu, Marius},
  title   = {Symmetry and spectral invariance for topologically graded $C^*$-algebras and partial action systems},
  journal = {Bull. London Math. Soc.},
  volume  = {54},
  year    = {2022},
  pages   = {1448--1469},
  doi = {10.1112/blms.12640},
}

@article{AO,
  author  = {Austad, Are and Ortega, Eduard},
  title   = {Groupoids and hermitian {B}anach $*$-algebras},
  journal = {Int. J. Math.},
  volume  = {33},
  year    = {2022},
  pages   = {2250090},
  doi = {10.1142/S0129167X22500902},
}

@article{Lagarias,
  author  = {Lagarias, Jeffrey C.},
  title   = {Geometric models for quasicrystals {I}. {D}elone sets of finite type},
  journal = {Discrete Comput. Geom.},
  volume  = {21},
  year    = {1999},
  pages   = {161--191},
  doi = {10.1007/PL00009413},
}

@incollection{BHZ,
  author    = {Bellissard, Jean and Herrmann, David and Zarrouati, Mohamed},
  title     = {Hulls of aperiodic solids and gap labelling theorems},
  booktitle = {Directions in Mathematical Quasicrystals},
  series    = {CRM Monogr. Ser.},
  volume    = {13},
  publisher = {Amer. Math. Soc.},
  year      = {2000},
}

@article{Kellendonk,
  author  = {Kellendonk, Johannes},
  title   = {Noncommutative geometry of tilings and gap labelling},
  journal = {Rev. Math. Phys.},
  volume  = {7},
  year    = {1995},
  pages   = {1133--1180},
}

@article{FGL,
  author  = {Fendler, Gero and Gr\"ochenig, Karlheinz and Leinert, Michael},
  title   = {Symmetry of weighted $L^1$-algebras and the {GRS}-condition},
  journal = {Bull. London Math. Soc.},
  volume  = {38},
  year    = {2006},
  pages   = {625--635},
  doi = {10.1112/S0024609306018777},
}

@article{GL,
  author  = {Gr\"ochenig, Karlheinz and Leinert, Michael},
  title   = {Symmetry and inverse-closedness of matrix algebras and functional calculus for infinite matrices},
  journal = {Trans. Amer. Math. Soc.},
  volume  = {358},
  year    = {2006},
  pages   = {2695--2711},
  doi = {10.1090/S0002-9947-06-03841-4},
}

@incollection{Groechenig,
  author    = {Gr\"ochenig, Karlheinz},
  title = {{W}iener's lemma: theme and variations. An introduction to spectral invariance and its applications},
  booktitle = {Four Short Courses on Harmonic Analysis},
  publisher = {Birkh\"auser},
  year      = {2010},
  editor = {Forster, Brigitte and Massopust, Peter},
  series = {Applied and Numerical Harmonic Analysis},
  pages = {175--234},
  doi = {10.1007/978-0-8176-4891-6_5},
}

@article{LP,
  author  = {Leptin, Horst and Poguntke, Detlev},
  title   = {Symmetry and nonsymmetry for locally compact groups},
  journal = {J. Funct. Anal.},
  volume  = {33},
  year    = {1979},
  pages   = {119--134},
}

@book{ADR,
  author    = {Anantharaman-Delaroche, Claire and Renault, Jean},
  title     = {Amenable Groupoids},
  series    = {Monogr. Enseign. Math.},
  volume    = {36},
  year      = {2000},
}

@book{Renault,
  author    = {Renault, Jean},
  title     = {A Groupoid Approach to $C^*$-Algebras},
  series    = {Lecture Notes in Math.},
  volume    = {793},
  publisher = {Springer},
  year      = {1980},
}

@article{BCFS,
  author  = {Brown, Jonathan and Clark, Lisa Orloff and Farthing, Cynthia and Sims, Aidan},
  title   = {Simplicity of algebras associated to \'etale groupoids},
  journal = {Semigroup Forum},
  volume  = {88},
  year    = {2014},
  pages   = {433--452},
}

@book{ExelBook,
  author    = {Exel, Ruy},
  title     = {Partial Dynamical Systems, Fell Bundles and Applications},
  series    = {Math. Surveys Monogr.},
  volume    = {224},
  publisher = {Amer. Math. Soc.},
  year      = {2017},
}

@article{SadunWilliams03,
  author  = {Sadun, Lorenzo and Williams, R. F.},
  title   = {Tiling spaces are {C}antor set fiber bundles},
  journal = {Ergodic Theory Dynam. Systems},
  volume  = {23},
  year    = {2003},
  pages   = {307--316},
}

@article{BBR,
  author  = {Benzi, Michele and Boito, Paola and Razouk, Nader},
  title   = {Decay properties of spectral projectors with applications to electronic structure},
  journal = {SIAM Rev.},
  volume  = {55},
  year    = {2013},
  pages   = {3--64},
}

@article{CO,
  author  = {Chen, Huajie and Ortner, Christoph},
  title   = {{QM/MM} methods for crystalline defects. {P}art 1: {L}ocality of the tight binding model},
  journal = {Multiscale Model. Simul.},
  volume  = {14},
  year    = {2016},
  pages   = {232--264},
}

@article{OTC,
  author  = {Ortner, Christoph and Thomas, Jack and Chen, Huajie},
  title   = {Locality of interatomic forces in tight binding models for insulators},
  journal = {ESAIM Math. Model. Numer. Anal.},
  volume  = {54},
  year    = {2020},
  pages   = {2295--2318},
  doi = {10.1051/m2an/2020020},
}

@article{MLO,
  author  = {Massatt, Daniel and Luskin, Mitchell and Ortner, Christoph},
  title   = {Electronic density of states for incommensurate layers},
  journal = {Multiscale Model. Simul.},
  volume  = {15},
  year    = {2017},
  pages   = {476--499},
}

@article{CLL,
  author  = {Canc\`es, \'Eric and Lahbabi, Salma and Lewin, Mathieu},
  title   = {Mean-field models for disordered crystals},
  journal = {J. Math. Pures Appl.},
  volume  = {100},
  year    = {2013},
  pages   = {241--274},
}

@book{PSB,
  author    = {Prodan, Emil and Schulz-Baldes, Hermann},
  title     = {Bulk and Boundary Invariants for Complex Topological Insulators},
  publisher = {Springer},
  year      = {2016},
}

@article{BK,
  author  = {Balan, Radu and Krishtal, Ilya},
  title   = {An almost periodic noncommutative {W}iener's lemma},
  journal = {J. Math. Anal. Appl.},
  volume  = {370},
  year    = {2010},
  pages   = {339--349},
}

@article{Kohn96,
  author  = {Kohn, W.},
  title   = {Density Functional and Density Matrix Method Scaling Linearly with the Number of Atoms},
  journal = {Phys. Rev. Lett.},
  volume  = {76},
  number  = {17},
  pages   = {3168--3171},
  year    = {1996},
  doi     = {10.1103/PhysRevLett.76.3168},
}

@article{PK05,
  author  = {Prodan, E. and Kohn, W.},
  title   = {Nearsightedness of Electronic Matter},
  journal = {Proc. Natl. Acad. Sci. USA},
  volume  = {102},
  number  = {33},
  pages   = {11635--11638},
  year    = {2005},
  doi     = {10.1073/pnas.0505436102},
}

@article{SS,
  author  = {Samei, E. and Shepelska, V.},
  title   = {Norm-Controlled Inversion in Weighted Convolution Algebras},
  journal = {J. Fourier Anal. Appl.},
  volume  = {25},
  number  = {6},
  pages   = {3018--3044},
  year    = {2019},
  doi     = {10.1007/s00041-019-09690-0},
}

@article{Nik,
  author  = {Nikolski, N.},
  title   = {In Search of the Invisible Spectrum},
  journal = {Ann. Inst. Fourier (Grenoble)},
  volume  = {49},
  number  = {6},
  pages   = {1925--1998},
  year    = {1999},
  doi     = {10.5802/aif.1743},
}

@book{BaakeGrimm13, author={Baake, Michael and Grimm, Uwe}, title={Aperiodic Order. Vol. 1: A Mathematical Invitation}, series={Encyclopedia of Mathematics and its Applications}, volume={149}, publisher={Cambridge University Press}, year={2013}}

@article{BaakeLenz04, author={Baake, Michael and Lenz, Daniel}, title={Dynamical systems on translation bounded measures: pure point dynamical and diffraction spectra}, journal={Ergodic Theory Dynam. Systems}, volume={24}, year={2004}, pages={1867--1893}}

@article{LagariasPleasants03, author={Lagarias, Jeffrey C. and Pleasants, Peter A. B.}, title={Repetitive Delone sets and quasicrystals}, journal={Ergodic Theory Dynam. Systems}, volume={23}, year={2003}, pages={831--867}}

@incollection{LenzStollmann03, author={Lenz, Daniel and Stollmann, Peter}, title={Delone dynamical systems and associated random operators}, booktitle={Operator Algebras and Mathematical Physics (Constan\c{t}a, 2001)}, publisher={Theta}, year={2003}, pages={267--285}}

@incollection{KellendonkPutnam00, author={Kellendonk, Johannes and Putnam, Ian F.}, title={Tilings, $C^*$-algebras and $K$-theory}, booktitle={Directions in Mathematical Quasicrystals}, series={CRM Monogr. Ser.}, volume={13}, publisher={Amer. Math. Soc.}, year={2000}, pages={177--206}}

@article{Hul, author={Hulanicki, A.}, title={On the spectrum of convolution operators on groups with polynomial growth}, journal={Invent. Math.}, volume={17}, year={1972}, pages={135--142}}

@article{SF, author={Shirali, S. and Ford, J. W. M.}, title={Symmetry in complex involutory {B}anach algebras. {II}}, journal={Duke Math. J.}, volume={37}, year={1970}, pages={275--280}}

@article{Ptak, author={Pt\'ak, V.}, title={Banach algebras with involution}, journal={Manuscripta Math.}, volume={6}, year={1972}, pages={245--290}}

@article{Exel, author={Exel, R.}, title={Amenability for {F}ell bundles}, journal={J. Reine Angew. Math.}, volume={492}, year={1997}, pages={41--73}}

@article{GK1, author={Gr\"ochenig, K. and Klotz, A.}, title={Noncommutative approximation: inverse-closed subalgebras and off-diagonal decay of matrices}, journal={Constr. Approx.}, volume={32}, year={2010}, pages={429--466}}

@article{GK2, author={Gr\"ochenig, K. and Klotz, A.}, title={Norm-controlled inversion in smooth {B}anach algebras, {I}}, journal={J. Lond. Math. Soc. (2)}, volume={88}, year={2013}, pages={49--64}}

@article{ENZ, author={El-Fallah, O. and Nikolski, N. K. and Zarrabi, M.}, title={Estimates for resolvents in {B}eurling--{S}obolev algebras}, journal={St. Petersburg Math. J.}, volume={10}, year={1999}, pages={901--964}}

@incollection{ShSu, author={Shin, C. E. and Sun, Q.}, title={Differential subalgebras and norm-controlled inversion}, booktitle={Operator Theory, Operator Algebras and Their Interactions with Geometry and Topology}, series={Oper. Theory Adv. Appl.}, volume={278}, publisher={Birkh\"auser}, year={2020}, pages={467--485},
  doi = {10.1007/978-3-030-43380-2_22},
}

@article{Deaconu95,
  author  = {Deaconu, Valentin},
  title   = {Groupoids associated with endomorphisms},
  journal = {Trans. Amer. Math. Soc.},
  volume  = {347},
  year    = {1995},
  pages   = {1779--1786},
}

@article{Abadie03,
  author  = {Abadie, Fernando},
  title   = {Enveloping actions and {T}akai duality for partial actions},
  journal = {J. Funct. Anal.},
  volume  = {197},
  year    = {2003},
  pages   = {14--67},
}

@article{BenedettiGambaudo03,
  author  = {Benedetti, Riccardo and Gambaudo, Jean-Marc},
  title   = {On the dynamics of {$G$}-solenoids. {A}pplications to {D}elone sets},
  journal = {Ergodic Theory Dynam. Systems},
  volume  = {23},
  year    = {2003},
  pages   = {673--691},
}

@article{Bost,
  author  = {Bost, Jean-Beno\^{\i}t},
  title   = {Principe d'{O}ka, {$K$}-th\'eorie et syst\`emes dynamiques non commutatifs},
  journal = {Invent. Math.},
  volume  = {101},
  year    = {1990},
  pages   = {261--333},
  doi = {10.1007/BF01231504},
}

@article{FGL08,
  author = {Fendler, Gero and Gr\"ochenig, Karlheinz and Leinert, Michael},
  title = {Convolution-dominated operators on discrete groups},
  journal = {Integral Equations Operator Theory},
  volume = {61},
  number = {4},
  year = {2008},
  pages = {493--509},
  doi = {10.1007/s00020-008-1604-7},
  eprint = {0801.0385},
  archivePrefix = {arXiv},
}

@article{Baskakov97,
  author = {Baskakov, A. G.},
  title = {Asymptotic estimates for the entries of the matrices of inverse operators and harmonic analysis},
  journal = {Siberian Math. J.},
  volume = {38},
  number = {1},
  year = {1997},
  pages = {10--22},
  doi = {10.1007/BF02674895},
}

@article{GKII,
  author = {Gr\"ochenig, Karlheinz and Klotz, Andreas},
  title = {Norm-controlled inversion in smooth {B}anach algebras, {II}},
  journal = {Math. Nachr.},
  volume = {287},
  number = {8--9},
  year = {2014},
  pages = {917--937},
  doi = {10.1002/mana.201200312},
  eprint = {1211.2974},
  archivePrefix = {arXiv},
}

@article{KB25,
  author = {K\"ohldorfer, Lukas and Balazs, Peter},
  title = {Wiener pairs of {B}anach algebras of operator-valued matrices},
  journal = {J. Math. Anal. Appl.},
  volume = {549},
  number = {2},
  year = {2025},
  eid = {129525},
  doi = {10.1016/j.jmaa.2025.129525},
  eprint = {2407.16416},
  archivePrefix = {arXiv},
}

@article{Flores25,
  author = {Flores, Felipe I.},
  title = {Polynomial growth and functional calculus in algebras of integrable cross-sections},
  journal = {J. Math. Anal. Appl.},
  volume = {549},
  number = {2},
  year = {2025},
  eid = {129486},
  doi = {10.1016/j.jmaa.2025.129486},
  eprint = {2401.09730},
  archivePrefix = {arXiv},
}

@article{AOP26,
  author = {Austad, Are and Ortega, Eduard and Palmstr\o m, Mathias},
  title = {{$K$}-theory invariance of {$L^p$}-operator algebras associated with \'etale groupoids of strong subexponential growth},
  journal = {Doc. Math.},
  year = {2026},
  note = {Published online first},
  doi = {10.4171/DM/1066},
  eprint = {2410.05005},
  archivePrefix = {arXiv},
}

@article{Nica08,
  author = {Nica, Bogdan},
  title = {Relatively spectral morphisms and applications to {$K$}-theory},
  journal = {J. Funct. Anal.},
  volume = {255},
  year = {2008},
  pages = {3303--3328},
  eprint = {0805.4788},
  archivePrefix = {arXiv},
}

@misc{Sims18,
  author = {Sims, Aidan},
  title = {Hausdorff \'etale groupoids and their {$C^*$}-algebras},
  year = {2018},
  eprint = {1710.10897},
  archivePrefix = {arXiv},
  primaryClass = {math.OA},
  note = {Version 2; Chapter 4}
}

\end{document}